\documentclass[twoside,a4paper,reqno,11pt]{amsart} 
\usepackage[top=28mm,right=30mm,bottom=28mm,left=30mm]{geometry}

\usepackage{amsfonts, amsmath, amssymb, mathrsfs, bm, latexsym, stmaryrd, array, hyperref, mathtools, bold-extra, xcolor}
\usepackage{verbatim}
\usepackage{float}
\restylefloat{table}

\newcommand{\la}{\langle}
\newcommand{\ra}{\rangle}

\newcommand{\leqs}{\leqslant}
\newcommand{\geqs}{\geqslant}

\newcommand{\vs}{\vspace{2mm}}

\newcommand{\Aut}{\operatorname{Aut}}

\newcommand{\Un}{{\rm U}}
\newcommand{\Li}{{\rm L}}

\makeatletter
\newcommand{\imod}[1]{\allowbreak\mkern4mu({\operator@font mod}\,\,#1)}
\makeatother

\newtheorem{theorem}{Theorem} 
\newtheorem*{conj*}{Conjecture}
\newtheorem*{thrm*}{Theorem}

\newtheorem*{nota}{Notation}
\newtheorem{thm}{Theorem}[section] 
\newtheorem{prop}[thm]{Proposition} 
\newtheorem{lem}[thm]{Lemma}
\newtheorem{cor}[thm]{Corollary}

\theoremstyle{definition}
\newtheorem{remark}{Remark}
\newtheorem{rmk}[thm]{Remark}

\newtheorem{defn}[thm]{Definition}

\newtheorem*{example*}{Example}

\begin{document}

\author{Emily V. Hall}
\address{E.V. Hall, School of Mathematics, University of Bristol, Bristol BS8 1UG, UK}
\email{ky19128@bristol.ac.uk}
\makeatletter
\@namedef{subjclassname@2020}{\textup{2020} Mathematics Subject Classification}
\makeatother


\begin{abstract}
Let $G$ be a transitive permutation group acting on a finite set $\Omega$ with $|\Omega|\geqs 2$. An element of $G$ is said to be a derangement if it has no fixed points on $\Omega$, and by a theorem of Jordan from 1872, $G$ contains such an element. In particular, by a theorem of Fein, Kantor and Schacher, $G$ contains a derangement of prime power order. Nevertheless there exist groups in which there are no derangements of prime order, these groups are called elusive groups. Defining a natural extension of this we say $G$ is almost elusive if it contains a unique conjugacy class of derangements of prime order. In recent work with Burness, we reduced the problem of determining the almost elusive quasiprimitive groups to the almost simple and 2-transitive affine cases. Additionally we classified the primitive almost elusive almost simple groups with socle an alternating group, a sporadic group or a group of Lie type with (twisted) Lie rank equal to 1. In this paper we complete the classification of the primitive almost elusive almost simple classical groups.
\end{abstract}

\title[]{Almost elusive classical groups} 
\maketitle
\section{Introduction}

Let $G\leqs{\rm Sym}(\Omega)$ be a transitive permutation group on a finite set $\Omega$ with $|\Omega|\geqslant 2$ and point stabiliser $H$. An element $x\in G$ is said to be a \emph{derangement} if it has no fixed points on $\Omega$, or equivalently if $x^G\cap H=\emptyset$ where $x^G$ denotes the conjugacy class of $x$ in $G$.
A classical theorem of Jordan \cite{Jordan} from 1872 guarantees the existence of derangements in $G$. This result has some interesting applications in number theory and topology, as discussed by Serre in \cite{serre}, and has led to various extensions in recent years. 
For example, in the 1980s Fein, Kantor and Schacher proved in \cite{FKS}, that derangements of prime power order always exist in finite transitive permutation groups. 
It is interesting to note that the proof given in \cite{FKS} requires the classification of finite simple groups, which is in clear contrast to the basic group theoretic concepts, such as the orbit-counting lemma, required to prove the existence of derangements.
 
Although the existence of prime power order derangements is guaranteed within transitive permutation groups, the existence of prime order derangements is not. 
We follow \cite{CGJKKMN} and call a transitive permutation group \emph{elusive} if it does not contain a derangement of prime order. These elusive permutation groups have been widely studied in recent years and have been the subject of many papers (for example \cite{AG,G,GK,GMPV,Xu}), although a complete classification of the transitive elusive groups is still to be completed. A large class of elusive groups was classified by Giudici in \cite{G}. Here it is shown that if $G$ is an elusive group with at least one transitive minimal normal subgroup, then there exists a positive integer $k$ such that $G={\rm M}_{11}\wr K$ acting with its product action on $\Delta^k$, where $|\Delta|=12$ and $K\leqs S_k$ is transitive. In particular, Giudici's theorem in \cite{G} shows that the only elusive almost simple primitive group is ${\rm M}_{11}$ with its action on 12 points (recall $G$ is almost simple if the socle $G_0$ of $G$ is a nonabelian finite simple group, giving $G_0\leqs G\leqs \Aut(G_0)$. Additionally a primitive group is a group in which all point stabilisers are maximal subgroups). 

In a different direction, one can consider the number of conjugacy classes of derangements. Since the number of fixed points of an element is unaffected by conjugation, the set of derangements in $G$ may be written as a union of conjugacy classes. The primitive permutation groups with one conjugacy class of derangements were classified by Burness and Tong-Viet in \cite{BTV1}. Here they show that $G$ has a unique conjugacy class of derangements if and only if $G$ is sharply 2-transitive or $(G,H)=(A_5,D_{10})$ or $(\Li_2(8){:}3,D_{18}{:}3)$. In later work of Guralnick \cite{Gur} it is shown that the same conclusion holds for all transitive groups.

In \cite{BHall} a natural extension of elusivity, combining conjugacy and derangements of prime order was introduced. We say a transitive permutation group acting on a set $\Omega$ is \emph{almost elusive} if there exists exactly one conjugacy class of derangements of prime order. We often say that the pair $(G,H)$ is almost elusive if $G$ is almost elusive with its action on the set of right cosets of $H$. A version of the O'Nan-Scott theorem for quasiprimitive groups, established by Praeger \cite{P93}, is used in \cite{BHall} to essentially reduce the problem of determining the almost elusive quasiprimitive groups to the almost simple and 2-transitive affine cases (recall that a finite permutation group is \emph{quasiprimitive} if every nontrivial normal subgroup is transitive). The classification of the almost elusive quasiprimitive groups was initiated in \cite{BHall}, in which the following result was shown;
\begin{thrm*}
Let $G\leqs{\rm Sym}(\Omega)$ be a primitive almost elusive almost simple permutation group with point stabiliser $H$. Additionally let the socle of $G$ be an alternating, a sporadic or a group of Lie type with (twisted) Lie rank equal to 1. Then $(G,H)$ is known. 
\end{thrm*}
For example if $n=p^a$ for some prime $p$, then $S_n$ with its natural action on $n$-points is almost elusive. This example is easy to see since any derangement of prime order must be the product of $n/p$ disjoint $p$-cycles, and these elements form a unique conjugacy class. We remark that in particular, there are infinitely many primitive almost simple almost elusive groups, which is in clear contrast to the elusive case. To complete the classification of primitive almost elusive groups it remains to handle the almost simple groups of Lie type with (twisted) Lie rank at least 2 and the affine groups. Our goal in this paper is to complete the classification for classical groups over $\mathbb{F}_q$, where $q=p^f$ with $p$ prime. 

Let $G\leqs {\rm Sym}(\Omega)$ be a primitive almost simple classical group over $\mathbb{F}_q$, with socle $G_0$ and point stabiliser $H$. We let $V$ denote the natural module of $G_0$ and write $\dim V=n$. A theorem by Aschbacher \cite{asch} regarding the subgroup structure of classical groups roughly states that $H$ is contained in one of nine subgroup collections. These collections consist of eight geometric collections (denoted as $\mathcal{C}_1,\dots,\mathcal{C}_8$): which include stabilisers of appropriate subspaces, direct sum and tensor product decompositions of $V$, and the non-geometric collection (denoted $\mathcal{S}$) containing the almost simple subgroups acting irreducibly on $V$.  We will often refer to the \emph{type} of $H$, for the geometric subgroups this gives an approximate description of the structure of $H\cap{\rm PGL}(V)$ (our usage is consistent with \cite[p.58]{KL}), and for non-geometric subgroups the \textit{type} of $H$ denotes the socle of $H$. For example, if $G_0=\Un_n(q)$ with $H$ of type ${\rm GU}_1(q)\perp {\rm GU}_{n-1}(q)$, then $H$ is the stabiliser of a 1-dimensional non-degenerate subspace of $V$. Note we adopt the standard notation and use $P_i$ to denote a maximal parabolic subgroup, this is the stabiliser in $G$ of an $i$-dimensional totally singular subspace of $V$.
We extensively use books by Kleidman and Liebeck \cite{KL}, and Bray, Holt and Roney-Dougal \cite{BHR}, for information on the structure and order of the point stabilisers.

Define $\mathcal{G}$ to be the set of finite simple classical groups over $\mathbb{F}_q$ of twisted Lie rank at least 2. 
We exclude from $\mathcal{G}$ any groups which are isomorphic to a finite simple group that is not a classical group of twisted Lie rank at least 2 (for example $\Li_4(2)\cong A_8$). Additionally, we exclude any duplicates (via isomorphism) of groups with twisted Lie rank at least 2 (for example ${\rm PSp}_4(3)\cong \Un_4(2)$, so we only include $\Un_4(2)$ in $\mathcal{G}$). Thus
\begin{equation}
\mathcal{G}=\{\Li_n(q)\mid n\geqs 3\}\cup\{\Un_n(q), {\rm PSp}_n(q)\mid n\geqs 4\}\cup\{{\rm P}\Omega^{\epsilon}_n(q)\mid n\geqs 7\}\setminus \mathcal{F}.
\end{equation}
where $\mathcal{F}=\{\Li_3(2),\Li_4(2),{\rm PSp}_4(2)^{'},{\rm PSp}_4(3)\}$. See \cite[Proposition 2.9.1]{KL} for a full list of isomorphisms between small dimensional classical groups.
\vspace{1.25cm}

\begin{theorem}\label{t:mainthrm}
Let $G\leqs{\rm Sym}(\Omega)$ be a finite almost simple primitive permutation group with classical socle $G_0\in\mathcal{G}$ and point stabiliser $H$. Then $G$ is almost elusive only if one of the following holds:
\begin{itemize}
\item[{\rm (i)}] $G_0= \Un_n(q)$, $H$ is of type ${\rm GU}_1(q)\perp {\rm GU}_{n-1}(q)$, $q$ is even and $n\geqs 5$ is a prime divisor of $q+1$. 
\item[{\rm (ii)}] $(G,H)$ is one of the cases recorded in Table \ref{tab:maintab}.
\end{itemize}
Moreover, each group appearing in Table \ref{tab:maintab} is almost elusive.
\end{theorem}

\begin{table}
\[
\begin{array}{lcllc} \hline
G_0 & &\mbox{Type of } H &G &  r \\ \hline\rule{0pt}{2.5ex} 
\Li_3(4) &\mathcal{C}_1 &   {\rm GL}_1(4)\oplus{\rm GL}_2(4)&\Li_3(4).2_1,\, \Li_3(4).2_3,\,\Li_3(4).2^2& 7\\
            &\mathcal{C}_5 &   {\rm GL}_3(2) &\Li_3(4).2_1,\, \Li_3(4).2_2,\,\Li_3(4).2^2&   5\\
            &\mathcal{S} &  A_6 &   \Li_3(4).2_3 &  7\\

\Un_6(2) &\mathcal{C}_5 &  {\rm Sp}_6(2) &\Un_6(2).2  & 11\\

             &\mathcal{S} & \Un_4(3) & \Un_6(2).2  & 11\\

\Un_5(2) &\mathcal{C}_1 &  {\rm GU}_1(2)\perp {\rm GU}_4(2) &\Un_5(2).2 &  11 \\
             &\mathcal{C}_2 & {\rm GU}_1(2)\wr S_5 & \Un_5(2).2& 11 \\

\Un_4(3) & \mathcal{C}_1 & P_2 & \Un_4(3).2_2 &  7 \\

\Un_4(2) &\mathcal{C}_1 &   P_1 & &  5\\
             &\mathcal{C}_2 & {\rm GU}_1(2)\wr S_4 & & 5\\
             &\mathcal{C}_5 & {\rm Sp}_4(2) & \Un_4(2).2 & 3 \\
{\rm PSp}_6(2) &\mathcal{C}_1 &  {\rm Sp}_2(2)\perp{\rm Sp}_4(2) &  & 7\\
                       &\mathcal{C}_8 & {\rm O}^{+}_6(2) &    &  3 \\
                       &\mathcal{C}_8 &  {\rm O}^{-}_6(2) &  &  7 \\
                
{\rm PSp}_4(7) &\mathcal{C}_1 & P_2 & & 5\\
\hline
\end{array}
\]
\caption{Primitive almost elusive groups with socle $G_0 \in \mathcal{G}$}
\label{tab:maintab}
\end{table}

In order to  make the following remark and state our next theorem (Theorem \ref{t:primedivs}), we introduce the notion of a \emph{primitive prime divisor}. We say that a prime divisor of $q^n-1$ is a primitive prime divisor if and only if it does not divide $q^i-1$ for all $0\leqs i<n$.
\begin{remark}\label{r:rmk1}
We anticipate that there are no almost elusive examples as in case (i) of Theorem \ref{t:mainthrm}. In later results (see Proposition \ref{p:uniprop}) we prove, if an almost elusive example appears in case (i) we additionally require that $q=2^f$ and $2nf+1$ is the unique primitive prime divisor of $q^{2n}-1$. In Section \ref{s:numthry} (in particular in Remark \ref{r:n=2aj}), we discuss how finding values of $f$ and $n$ for which these conditions are satisfied comes down to being able to solve a particular Diophantine equation, which currently does not have a complete set of integer solutions. In addition, we can use Lemma \ref{l:prime fac of f} to deduce that $f=2^an^b$ for some integers $a,b>0$, and with the aid of a computer we then deduce that $f,n>100$ for an almost elusive group to appear in case (i). 
\end{remark}

\begin{remark}\label{r:tabrmk}
Here we provide some useful remarks regarding Table \ref{tab:maintab}
\begin{itemize}\addtolength{\itemsep}{0.2\baselineskip}
\item[{\rm (a)}] The fourth column in Table \ref{tab:maintab}, labeled $G$, lists the groups $G$ for which $(G,H)$ is almost elusive. If there is no entry in this column, it implies that $(G,H)$ is almost elusive for all $G$ with socle $G_0$. The fifth column in Table \ref{tab:maintab} is labeled $r$, this denotes the order of the elements in the unique conjugacy of derangements of prime order. 

\item[{\rm (b)}] For the cases with $G_0=\Li_3(4)$ in Table \ref{tab:maintab} the recorded groups $G$ in the fourth column are defined using Atlas notation \cite{WebAt}. That is $\Li_3(4).2_1=\Li_3(4).\langle \iota\phi \rangle $, $\Li_3(4).2_2=\Li_3(4).\langle \phi \rangle$, $\Li_3(4).2_3=\Li_3(4).\langle \iota\rangle$ and $\Li_3(4).2^2=\Li_3(4).\langle \iota,\phi \rangle$, where $\iota$ denotes the inverse-transpose graph automorphism and $\phi$ denotes a field automorphism of order $2$. Similarly for the case with $G_0=\Un_4(3)$ in Table \ref{tab:maintab}. Here $\Un_4(3).2_2=\Un_4(3).\langle \gamma \rangle $, where $\gamma$ is an involutary graph automorphism and the centraliser of $\gamma$ in $\Un_4(3)$ is ${\rm PSp}_4(3)$.
\end{itemize}
\end{remark}

We recall that $x\in G$ is a derangement if and only if $x^G\cap H=\emptyset$. Since $G_0$ is a normal subgroup of $G$, it is easy to see that $G$ is almost elusive only if $\pi(G_0)-\pi(H_0)\leqs 1$, where $H_0:=H\cap G_0$ and $\pi(X)$ denotes the number of distinct prime divisors of $|X|$. The subgroups $M$ of a simple group $G_0$ with $\pi(G_0)=\pi(M)$ are described by Liebeck, Praeger and Saxl in \cite[Corollary 5]{LPS}. Here we extend this result for maximal subgroups of classical groups. Define $\widetilde{\mathcal{G}}$ to be the set of all finite simple classical groups over $\mathbb{F}_q$, not including duplicates obtained via isomorphisms (for example $\Li_2(7)\cong \Li_3(2)$ so we include $\Li_2(7)$ only). For the following theorem see Section \ref{s:primedivs} for the tables and Remark \ref{r:pi tables} for more details.

\begin{theorem} \label{t:primedivs}
Let $G\leqs {\rm Sym}(\Omega)$ be a finite almost simple primitive permutation group with point stabiliser $H$ and socle $G_0\in \widetilde{\mathcal{G}}$. Then $\pi(G_0)\leqs \pi(H_0)+1$ if and only if one of the following holds;
\begin{itemize}
\item[{\rm (i)}] $\pi(G_0)=\pi(H_0)$ and $(G_0,H)$ is one of the cases recorded in Table \ref{tab:pi0}; or
\item[{\rm (ii)}] $\pi(G_0)=\pi(H_0)+1$ and either $(G_0,H)$ is found in Table \ref{tab:pi2}, or $(G_0,H,i)$ is one of the cases recorded in Table \ref{tab:pi1} and there exists a unique primitive prime divisor of $q^i-1$.
\end{itemize}
\end{theorem}
In Section 4 we provide a proof of Theorem \ref{t:primedivs}. This essentially provides us with a reduction for the cases to be handled in the proof of Theorem \ref{t:mainthrm}, to the cases recorded in Tables \ref{tab:pi0}, \ref{tab:pi2} and \ref{tab:pi1}. We organise the proof in Section 4 via Aschbacher's subgroup structure theorem (discussed in Section \ref{s:subgrp struc}). For the geometric subgroups we approach the proof by a direct comparison of $|G_0|$ and $|H_0|$, which are both presented in \cite{KL}. Our main approach here is to search for primitive prime divisors of integers of the form $q^i-1$ which divide $|G_0|$ but not $|H_0|$. In some cases the size and uniqueness of these primitive prime divisors is required to complete the proof. For this we rely on some new number theoretic results in Section 2 on primitive prime divisors (see Lemma \ref{l:ppd2a3} for example), which we hope may also have independent applications in other areas of research. The non-geometric subgroups are slightly harder to handle since it is not possible to list all such subgroups for a classical group in general. In this case for $n\geqs 13$ we use a result of Guralnick et al. in \cite{GPPS}, which allows us to easily find distinct primitive prime divisors that divide $|G_0|$ but not $|H_0|$ in the majority of cases. When $n\leqs 12$ we use similar techniques to the geometric subgroups, using the tables in \cite{BHR} to read off $|H_0|$.

We complete the proof of Theorem \ref{t:mainthrm} in the final section, Section 5. Some small groups that arise can be handled directly with the aid of {\sc Magma} \cite{Mag} (see Proposition \ref{p:smalldim}). 
In the general setting we aim to construct derangements of distinct prime order (for cases in Table \ref{tab:pi1}, only one other prime order derangement needs to be found). Additionally, we often use results on conjugacy classes (see Section 3.2) combined with results on primitive prime divisors (see Section 2) to show the number of $G$-classes of derangements of a particular prime order. The subspace subgroups prove to be the trickiest cases to handle and in some instances require special attention. In particular, the cases in which $H$ is the stabiliser of a non-degenerate 1-space prove to be the hardest to handle (this is not surprising since these are relatively large subgroups). In fact it is precisely a subgroup of this type for $G_0=\Un_n(q)$ that leads to the special case outlined in part (i) of Theorem \ref{t:mainthrm} (see Proposition \ref{p:uniprop} for more details).

We will complete the classification of the primitive almost elusive permutation groups in a sequel, where we will handle the affine groups and the remaining exceptional groups of Lie type. 

\vs

\noindent \textbf{Notation.} Our group theoretic notation is fairly standard. Let $A$ and $B$ be groups and $n$ a positive integer. We write $C_n$ or $n$ to denote a cyclic group of order $n$ and $[n]$ to denote an unspecified soluble group of order $n$. An unspecified extension of $A$ by $B$ will be denoted as $A.B$, and we use $A{:}B$ if the extension splits. For positive integers $a$ and $b$ we use $(a,b)$ to denote the greatest common divisor of $a$ and $b$. Additionally we use $(a)_b$ to denote the largest power of $b$ dividing $a$. For simple groups we adopt the notation of Kleidman and Liebeck \cite{KL}. For instance, 
$$ {\rm PSL}_n(q)=\Li_n(q),\,\,\,{\rm PSU}_n(q)=\Un_n(q). $$
If $G$ is a simple orthogonal group then we write $G={\rm P}\Omega^{\epsilon}_n(q)$ where $\epsilon=\circ$ if $n$ is odd and $\epsilon=-$ (respectively $+$) if $n$ is even and the underlying quadratic form has Witt defect 1 (respectively 0). In the case $n$ is odd we often write $G=\Omega_n(q)$.

\vs

\noindent \textbf{Acknowledgments.}
I would like to thank my Ph.D. supervisor Tim Burness for his guidance and helpful comments. I would also like to acknowledge the financial support of the Heilbronn Institute for Mathematical Research.

\section{Number theoretic preliminaries}\label{s:numthry}
In this section, we present several number-theoretic results that will be needed in the proofs of both Theorem \ref{t:mainthrm} and Theorem \ref{t:primedivs}. Throughout the section we let $n$ be a positive integer and $q=p^f$ such that $p$ is prime and $f\geqs 1$. Our first result is \cite[Lemma 2.6]{BTV}.

\begin{lem}\label{l:btv}
Let $r$ and $s$ be primes and let $v$ and $w$ be positive integers. If $r^v + 1 =s^w$ then one of the following holds:
\begin{itemize}\addtolength{\itemsep}{0.2\baselineskip}
    \item[{\rm (i)}] $(r,s,v,w)=(2,3,3,2)$.
    \item[{\rm (ii)}] $(r,w)=(2,1)$ and $s=2^v+1$ is a Fermat prime. 
    \item[{\rm (iii)}] $(s,v)=(2,1)$ and $r=2^w-1$ is a Mersenne prime. 
\end{itemize}
\end{lem}

\begin{lem}\label{l:Lemma A.4}
Let $r$ be a prime divisor of $q-\epsilon$ where $\epsilon =\pm 1$. Then 
\begin{equation*}
(q^n-\epsilon)_r=
\begin{cases}
(q-\epsilon)_r(n)_r & n \mbox{ odd, or } r \mbox{ odd and } \epsilon=+ 1\\
(q^2-1)_2(n)_2/2 & n\mbox { even, } r=2, \epsilon =+ 1\\
(r,2) & n \mbox{ even, }\epsilon=- 1
\end{cases}
\end{equation*}
\end{lem}
\begin{proof}
This is \cite[Lemma A.4]{BG_book}
\end{proof}

In the remainder of the section we provide results regarding primitive prime divisors, which are an invaluable tool in the proofs of Theorems \ref{t:mainthrm} and \ref{t:primedivs}. Let $a\geqs 2$ be an integer. We say a prime divisor of $a^n-1$ is a \emph{primitive prime divisor} (of $a^n-1$) if it does not divide $a^i-1$ for all $0\leqs i<n$. We define 
$$P_a^n=\{r\mid r \mbox{ is a primitive prime divisor of } a^n-1\}.$$ 
A theorem of Zsigmondy \cite{Zsig} from the 1890s, states that $P_a^n\not=\emptyset$ unless either $(n,a)=(1,2),(6,2)$, or $n=2$ and $a=p$ is a Mersenne prime.

The following result has an elementary proof. Details of the first part can be found in \cite[Lemma A.1]{BG_book}, and the second is an easy consequence of Fermat's Little Theorem. 

\begin{lem}
\label{l:Lemma A.1}
Assume that $r\in P_a^n$ is an odd prime and let $m$ be a positive integer. Then $r$ divides $a^m-1$ if and only if $n$ divides $m$. Additionally $r\equiv 1\imod n $.
\end{lem}

For the purposes of this paper we are mainly interested in finding the size of unique primitive prime divisors. In particular, for which $n$, $q$ and $d$ do we get $P_q^n=\{dn+1\}$. The remainder of this section is dedicated to discussing this problem.

The following lemma provides a connection between primitive prime divisors of $p^{fn}-1$ and $q^n-1$, which will be useful in our discussion on unique primitive prime divisors. We state the lemma in more general setting.

\begin{lem}\label{l:prime fac of f}
Let $a,b$ and $c$ be positive integers such that $b\geqs 2$ and $a=b^c$. Let $n\geqs 2$ with prime factorisation $n=s_1^{g_1}\dots s_t^{g_t}$ where the $s_i$ are distinct primes and each $g_i$ is a positive integer. Then $P_b^{cn}\subseteq P_a^n$, with equality if and only if one of the following holds:
\begin{itemize}
\item[{\rm (i)}] $(n,b)=(6,2)$ and $c$ is prime;
\item[{\rm (ii)}] $n=2$, $c$ is prime and $b$ is a Mersenne prime; or 
\item[{\rm (iii)}] $c=s_1^{h_1}\dots s_t^{h_t}$ with every $h_i\geqs 0$.
\end{itemize}
Moreover $|P_a^n|=1$ only if ${\rm (i),(ii)}$, or ${\rm (iii)}$ holds, or $(n,c,b)=(3,2,2),(2,3,2)$.
\end{lem}
\begin{proof}
Assume $r\in P_b^{cn}$. Then by definition $r$ divides $b^{cn}-1$ but does not divide $b^i-1$ for all $0\leqs i <cn$. Thus it is easy to see that $r\in P_a^n$, so it follows that $P_b^{cn}\subseteq P_a^n$.  In order to prove the first part of the lemma it remains to prove the equality condition. Equality is clear for $c=1$ so for the remainder of the proof we may assume $c\geqs 2$.

Write $c=mk$ where $m=s_1^{h_1}\dots s_t^{h_t}$ with all $h_i\geqs 0$, and $k\geqs 1$ with $(k,n)=1$. From here we define three separate cases:

\begin{itemize}
\item[{\rm (a)}] $k>1$, $(n,m,b)\neq(6,1,2)$, and $(n,m)\neq(2,1)$ when $b$ is a Mersenne prime. 
\item[{\rm (b)}] $k>1$ and $(n,m,b)=(6,1,2)$, or $(n,m)=(2,1)$ and $b$ is a Mersenne prime.
\item[{\rm (c)}] $k=1$.
\end{itemize}

First consider case (a). Define $v=mn$ and take $r\in P_b^v$ (note that our assumptions on $n,m$ and $b$ imply that such an $r$ always exists). Since $v<cn$, by definition $r\not\in P_b^{cn}$. However $v$ divides $cn$ but not $cd=mkd$ for any $0\leqs d<n$, implying $r\in P_a^n$ by Lemma \ref{l:Lemma A.1}. Therefore $P_b^{cn}\neq P_a^n$.

Next let us turn to case (b). Here $c=k$ and we let $c=p_1^{x_1}\dots p_l^{x_l}$ be the prime factorisation of $c$. 
Suppose first that $c$ is composite and take $r\in P_b^{p_1n}$. Then by definition $r\not\in P_b^{cn}$ since $p_1n<cn$. However $p_1n$ divides $cn$, but not $cd$ for all $0\leqs d<n$ since $(c,n)=1$. Therefore $r\in P_a^n$, and so $P_b^{cn}\neq P_a^n$. 
Finally suppose that $c$ is prime. Take $r$ to be a prime divisor of $a^n-1=b^{cn}-1$ such that $r\not\in P_b^{cn}$. Then $r\in P_b^j$ for some $1\leqs j<cn$. Suppose $j=1$, then $r$ divides $b^c-1=a-1$ implying $r\not\in P_a^n$. Now suppose that $j\geqs 2$ and note that $j$ divides $cn$ by Lemma \ref{l:Lemma A.1}. Thus since $c$ is prime, either $j$ divides $cd$ for some $0\leqs d<n$ or $j=n$. However $P_b^n=\emptyset$ by Zsigmondy's theorem implying that $j\neq n$. Therefore $r\not\in P_a^n$ by Lemma \ref{l:Lemma A.1}, so $P_a^n\subseteq P_b^{cn}$ and in particular $P_b^{cn}=P_a^n$.

Finally let us assume that $k=1$ as in case (c). Here $c=m=s_1^{h_1}\dots s_t^{h_t}$. As for case (b) we take $r$ to be a prime divisor of $a^n-1=b^{cn}-1$ such that $r\not\in P_b^{cn}$. Then $r\in P_b^j$ for some $1\leqs j<cn$ such that $j$ divides $cn$. 
Thus $j=s_1^{w_1}\dots s_t^{w_t}$ where $w_i\leqs g_i+h_i$ (note that since $j\neq cn$ we must have that $w_i<g_i+h_i$ for at least one value of $i$). Define $d=s_1^{z_1}\dots s_t^{z_t}$ where 
\[
z_i=
\begin{cases}
w_i-h_i & \mbox{if } w_i-h_i>0\\
0 & \mbox{otherwise}.
\end{cases}
\]
Then $d$ divides $n$ since $z_i\leqs g_i$ for all $i$, and in particular $d\neq n$ since $j\neq cn$. By construction $j$ divides $cd$, so $r\not\in P_a^n$. Thus equality holds.

For the remaining part of the lemma it is easy to see that $|P_a^n|=1$ only if $P_b^{cn}=P_a^n$ or $|P_b^{cn}|=0$. Thus the result follows. 
\end{proof}

We now begin our discussion on unique primitive prime divisors of $q^n-1$ ($q=p^f$ is a prime power). Assume that $|P_q^n|=1$. Then for $n\geqs2$, Lemma \ref{l:prime fac of f} tells us that $P_q^n$ contains every primitive prime divisor of $P_p^{fn}$. Thus if $P_p^{fn}\neq\emptyset$, using Lemma \ref{l:Lemma A.1} we obtain a lower bound for the unique primitive prime divisor $r$ of $q^n-1$, that is $r\geqs nf+1$. We dedicate the remainder of this section to improving this bound on $r$ for certain values of $n$.

\begin{lem}\label{l:ppds4}
Let $n=2^a$ for some $a\geqs 2$. If $P_q^n=\{r\}$ then either $r\geqs 4nf+1$, or one of the following holds;
\begin{itemize}
\item[{\rm (i)}] $(n,q)=(4,2),(4,3),(4,7)$ and $r=nf+1=5$; 
\item[{\rm (ii)}] $(n,q)=(4,4),(8,2)$ and $r=2nf+1=17$; or
\item[{\rm (iii)}] $(n,q)=(4,5),(4,239)$ and $r=3nf+1=13$.
\end{itemize}
\end{lem}
\begin{proof}
Suppose $P_q^n=\{r\}$. Since $P_q^n$ contains every primitive prime divisor of $p^{nf}-1$, it follows that $r=dnf+1$ for some $d\geqs 1$ and so we may assume that $d=1,2$ or $3$. Using Lemma \ref{l:Lemma A.1} it is easy to see that a prime $s$ is an element of $P_q^n$ if and only if $s$ is an odd prime divisor of $q^{n/2}+1$. 

First let us assume that $q$ is even. Then $q^{n/2}+1=r^l$ for some $l\geqs 1$. By Lemma \ref{l:btv} we must have $l=1$ and $r$ is a Fermat prime, that is $2^{nf/2}+1=dnf+1$. It is straightforward to show that for $d=3$ there are no solutions, for $d=2$ the only solutions are $(n,q)=(4,4),(8,2)$, and similarly for $d=1$ the only solution is $(n,q)=(4,2)$.

Now assume that $q$ is odd, it follows that $q^{n/2}+1=2r^l$ for some $l\geqs 1$. Then \cite[Lemma 2.6]{ZLT} tells us that if $l\geqs 3$ then $(n,q)=(4,239)$ (note here that $(r,l)=(13,4)$). Therefore we may now assume that $l=1$ or $l=2$ and thus
$$p^{\frac{1}{2}nf}+1=2(dnf+1) \;\mbox{ or }\; p^{\frac{1}{2}nf}+1=2(dnf+1)^2$$ 
for $d=1,2$ or $3$. From here it is straightforward to show that $(n,q)=(4,3),(4,5),(4,7)$ are the only solutions. Note that $(r,d,l)=(5,1,1)$ for $(n,q)=(4,3)$, $(r,d,l)=(13,3,1)$ for $(n,q)=(4,5)$ and $(r,d,l)=(5,1,2)$ for $(n,q)=(4,7)$. The result follows. 
\end{proof}

\begin{rmk}
We note that in Lemma \ref{l:ppds4} the case $n=2$ is not handled. It is not difficult to see that $r\in P_q^2$ if and only if $r$ is an odd prime divisor of $q+1$. Thus if $P_q^2=\{r\}$ then for some positive integers $k$ and $l$ either $q$ is even and $q+1=r^l$, or $q$ is odd and $q+1=2^kr^l$. The case when $q$ is even can be handled using Lemma \ref{l:btv}, showing that either $(q,r)=(8,3)$ or $q+1$ is a Fermat prime. However, the case when $q$ is odd leaves a much harder Diophantine equation to solve. In this case we were unable to obtain a full solution, although Lemma \ref{l:prime fac of f} does provide restrictions on $q$. In particular if $P_q^2=\{r\}$ for $q$ odd, then either $q=9$ (in which case $r=5$), $f=2^m$ for some $m\geqs 0$, or $p$ is a Mersenne prime with $f$ prime. 
\end{rmk}

We will now state some results from the literature that will be useful in the proofs of our number-theoretic results. The first is a theorem of Nagell \cite{Nag}.

\begin{thm}\label{t:nag}
Let $p\geqs 3$ be a prime. The only integer solutions to the equation $$x^2+x+1=3y^p$$ are $(x,y)=(1,1)$ and $(x,y)=(-2,1)$.
\end{thm}

\begin{prop}\label{p:blnag}
Let $x,y,a$ and $b$ be integers such that $|x|,|y|>1$, $a>2$ and $b\geqs 2$. Suppose $(x,y,a,b)$ is a solution to 
\[
\frac{x^a-1}{x-1}=y^b,
\]
such that $(x,y,a,b)\neq (3,11,5,2),(7,20,4,2),(18,7,3,3)$ or $(-19,7,3,3)$. Then the following hold
\begin{itemize}
\item[{\rm (i)}] $b\geqs 3$ is prime.
\item[{\rm (ii)}] The least prime divisor $r$ of $a$ satisfies $r\geqs 5$.
\item[{\rm (iii)}] $|x|\geqs 10^4$ and $x$ has a prime divisor $r\equiv 1\imod b$.
\end{itemize}
\end{prop}
\begin{proof}
This is \cite[Proposition 1]{BL}
\end{proof}

The following lemma is a generalisation of \cite[Lemma 4.14]{BHall}.

\begin{lem}\label{l:ppd2a3}
Let $n=2^a3$ for some $a\geqs 0$. If $P_q^n=\{r\}$ then either $r\geqs 4nf+1$, or one of the following holds
\begin{itemize}
\item[{\rm (i)}] $(n,q)=(3,4),(6,3),(6,4),(6,5),(6,8),(6,19),(12,2)$ and $r=nf+1$; or
\item[{\rm (ii)}] $(n,q)=(3,2),(6,23)$ and $r=2nf+1$; 
\end{itemize}
\end{lem}
\begin{proof}
Suppose $P_q^n=\{r\}$. Since $P_q^n$ contains every primitive prime divisor of $p^{nf}-1$, it follows that $r=dnf+1$ for some $d\geqs 1$ and so we may assume that $d=1,2$ or $3$. We split the proof into two separate cases namely, $a=0$ and $a\geqs 1$. We note that the proofs of both cases are similar but we provide the details for completeness.

Assume first that $a=0$, that is $n=3$. In this case $r$ divides $q^2+q+1$. If $s\geqs 5$ is a prime divisor of $q^2+q+1$ it is easy to check that $s$ does not divide $q-1$, thus $s$ is a primitive prime divisor of $q^3-1$ implying that $r=s$. Since $q^2+q+1$ is indivisible by 9 and odd it follows that either $q^2+q+1=r^l$, or $q\equiv 1\imod 3$ and $q^2+q+1=3r^l$ for some positive integer $l$.

Suppose $q\equiv 1\imod 3$ and $q^2+q+1=3r^l$. By Theorem \ref{t:nag}, if $l\geqs 3$ then there are no integer solutions $(q,r)$. Thus we may assume $l=1$ or $2$, so
$$p^{2f}+p^f+1=3(3df+1)\mbox{ or } p^{2f}+p^f+1=3(3df+1)^2.$$
It is straightforward to check that $q=4$ is the only possibility, with $r=7$ and $q^2+q+1=3r$.

Finally suppose $q\not\equiv 1\imod 3$ and $q^2+q+1=r^l$. If $l\geqs 2$ then by applying Propositon \ref{p:blnag} we deduce that there are no solutions, so we may assume $l=1$. It is straightforward to check that $q=2$ is the only possibility, with $r=7$ and $q^2+q+1=r$.

Now let us assume that $a\geqs 1$. In a similar manner to the $a=0$ case it is easy to show that if $s\geqs 5$ is a prime divisor of $q^{2^a}-q^{2^{a-1}}+1$, then $s=r$.  Since $q^{2^a}-q^{2^{a-1}}+1$ is odd and indivisible by 9 (in particular indivisible by 3 when $a\geqs 2$), it follows that either $q^{2^a}-q^{2^{a-1}}+1=r^l$, or $a=1$ with $q\equiv 2\imod 3$ and $q^{2^a}-q^{2^{a-1}}+1=3r^l$ for some positive integer $l$.

Suppose $a=1$ with $q\equiv 2\imod 3$ and $q^2-q+1=3r^l$. By Theorem \ref{t:nag}, if $l\geqs 3$ then there are no integer solutions $(q,r)$, so we may assume $l=1$ or $2$ that is,
$$p^{2f}-p^f+1=3(6df+1)\mbox{ or } p^{2f}-p^f+1=3(6df+1)^2.$$
It is straightforward to check that $(q,r)=(5,7),(8,19),(23,13)$ are the only solutions.

Finally suppose $q^{2^a}-q^{2^{a-1}}+1=r^l$. Setting $x=-q^{2^{a-1}}$ we get an integer solution to the equation $x^2+x+1=r^l$. By applying Proposition \ref{p:blnag}, if $l\geqs 2$ the only solution is $(x,r,l)=(-19,7,3)$, that is $n=6$, $q=19$, $r=7$ and $q^2-q+1=r^3$. Therefore we may now assume that $l=1$. From here it is straightforward to show that for $l=1$ the only solutions are $(n,q)=(6,3)$ with $r=nf+1=7$ or $(n,q)=(6,4),(12,2)$ with $r=nf+1=13$.
\end{proof}

\begin{rmk}\label{r:n=2aj}
As the prime decomposition of $n$ becomes more complicated it becomes increasingly more difficult to find the size of a unique primitive prime divisor of $q^n-1$. In particular, the associated Diophantine equation becomes more challenging to both find and solve. Here we discuss some of the issues faced for $n=2^aj$ such that $j\geqs5$ is an odd prime and $a\geqs 0$. (Note the case $j=3$ was handled in Lemma \ref{l:ppd2a3}).

If there exists a unique primitive prime divisor $r$ of $q^n-1$, then $(n,q,r)$ must be a solution to 
\begin{equation*}
\frac{q^{2^{a-1}j}+1}{q^{2^{a-1}}+1}=(j,q^{2^{a-1}}+1)r^l \,\,\mbox{ if } a\geqs 1 \mbox{ or }
\end{equation*} 
\begin{equation*}
\frac{q^{n}-1}{q-1}=(n,q-1)r^l \,\, \mbox{ if } a=0
\end{equation*} 
for some positive integer $l$. These are both particular examples of the general \textit{Nagell-Ljunggren equation}; for which there currently does not exist a complete set of integer solutions. However bounds on the potential solutions have been established (see \cite{Mih}). 
This means that for $n=2^aj$ with $j\geqs 5$ we are unable to provide a lemma as general as Lemmas \ref{l:ppds4} and \ref{l:ppd2a3}. Nevertheless we can give some details in the cases when $q=2^f$ for some positive integer $f$ and $a=0$ (see Lemma \ref{l:lic1prime}) or $a=1$ (see Lemma \ref{l:uniq ppd}). 
\end{rmk}

Recall that for positive integers $a$ and $b$, the notation $(a)_b$ denotes the largest power of $b$ dividing $a$. 

\begin{lem}\label{l:lic1prime}
Let $n\geqs 5$ be an odd prime and $q=2^f$ for some positive integer $f$. If $P_q^n=\{r\}$ then $r\geqs 4nf+1$.
\end{lem}
\begin{proof}
Suppose $P_q^n=\{r\}$. Then in the usual manner, $r$ is the unique primitive prime divisor of $2^{fn}-1$ and so $r=dnf+1$ for some $d\geqs 1$. By Lemma \ref{l:prime fac of f} we know that $f=n^j$ for some $j\geqs 0$. Thus it follows that $d$ must be even since both $r$ and $n$ are odd primes, so we may assume $r=2nf+1$, that is $r=2n^{j+1}+1$.

Assume $n$ divides $2^f-1=2^{n^j}-1$. Then by Lemma \ref{l:Lemma A.1}, since $n$ is prime, $n$ is a primitive prime divisor of $2^{n^t}-1$ for some $1\leqs t\leqs j$. This implies that $n\equiv 1\imod {n^t}$ by Lemma \ref{l:Lemma A.1}, which is an obvious contradiction. Thus for any prime divisor $k$ of $2^f-1$, we conclude that $(2^{fn}-1)_k=(2^f-1)_k$ (see Lemma \ref{l:Lemma A.4}).

Suppose $s$ is a prime divisor of $(2^{fn}-1)/(2^f-1)$. By the above argument it follows that $s$ divides $2^{fn}-1=2^{n^{j+1}}-1$ but does not divide $2^f-1=2^{n^j}-1$. Therefore using Lemma \ref{l:Lemma A.1} once again, we see that $s$ is a primitive prime divisor of $2^{fn}-1$.

In particular we conclude that,
\begin{equation}\label{e:2fn}
\frac{2^{fn}-1}{2^f-1}=(2nf+1)^l
\end{equation}
for some positive integer $l$. By applying Proposition \ref{p:blnag} we see there are no solutions to (\ref{e:2fn}) when $l\geqs 2$. Additionally, it is straightforward to show that for $l=1$ there are also no solutions. Thus the result follows.
\end{proof}

\begin{lem}\label{l:uniq ppd}
Let $n\geqs 5$ be an odd prime and $q=2^f$ for some positive integer $f$. If $P_{q}^{2n}=\{r\}$ then either $r\geqs 4nf+1$, or $r=2nf+1$ and one of the following holds:
\begin{itemize}
\item[{\rm (i)}] $(n,q)=(5,2)$; or
\item[{\rm (ii)}] $n$ divides $q+1$. 
\end{itemize}
Moreover, if $P_{q}^{2n}=\{r\}$ then $f=2^an^b$ for some integers $a,b\geqs 0$.
\end{lem}
\begin{proof}
Suppose $P_{q}^{2n}=\{r\}$. Again in the usual manner we have that $r=2nfd+1$ for some $d\geqs 1$, so we may assume $d=1$.

Assume $s$ is a prime divisor of $q^{2n}-1$. Then since $n$ is an odd prime, either $s$ is a divisor of $q^2-1$ or $s$ is a primitive prime divisor of $q^n-1$ or $q^{2n}-1$. Thus any prime divisor of $q^n+1$ is either a primitive prime divisor of $q^{2n}-1$ or is a divisor of $q+1$. By Lemma \ref{l:Lemma A.4} it follows that
\begin{equation}\label{e:q2n-1}
\frac{q^{n}+1}{q+1}=(n,q+1)(2nf+1)^l
\end{equation}
for some positive integer $l$. 

Suppose that $n$ does not divide $q+1$. Then $(n,q+1)=1$ and by applying Proposition \ref{p:blnag} we see there are no solutions to \eqref{e:q2n-1} for $l\geqs 2$. It remains to deal with the case when $l=1$. Here it is straightforward to show that $(n,q)=(5,2)$ is the only solution (note here that $r=11$). This concludes the first part of the lemma. 
The final part of the lemma is a straightforward application of Lemma \ref{l:prime fac of f}.
\end{proof}
 
\section{Classical groups}

In this section we set up notation and record several preliminary results for classical groups that will be useful in the proofs of our main theorems. In particular, we will briefly discuss Aschbacher's subgroup structure theorem and recall information on prime order conjugacy classes. As before we let $G\leqs {\rm Sym}(\Omega)$ be a primitive almost simple classical group over $\mathbb{F}_q$ with socle $G_0$ and point stabiliser $H$. Let $V$ denote the natural module of $G_0$ such that $n=\dim(V)$ and let $q=p^f$ for some prime $p$ and positive integer $f$.

\subsection{Subgroup structure}\label{s:subgrp struc}
 
Recall that a point stabiliser in a primitive permutation group is a maximal subgroup. The main theorem we use on the subgroup structure of finite classical groups is due to Aschbacher. In \cite{asch}, he defines nine separate collections of subgroups. The first eight collections (denoted $\mathcal{C}_1,\dots,\mathcal{C}_8$) are often referred to as the \textit{geometric} collections: the members of these collections are defined in terms of the underlying geometry of the natural module $V$. For example, the $\mathcal{C}_2$ collection consists of the stabilisers of an appropriate direct sum decomposition of $V$. The ninth collection (denoted $\mathcal{S}$) is often called the \emph{non-geometric} collection, and consists of almost simple groups that act absolutely irreducibly on $V$ (see \cite[p.3]{KL} for a formal definition of this collection). Note we adopt the definition of these collections used by Kleidman and Liebeck in \cite{KL}, which differs slightly from the original set up in \cite{asch}. A brief description of each of these collections are outlined in Table \ref{tab:geom}. We write $\mathcal{C}=\mathcal{C}_1\cup\dots\cup\mathcal{C}_8$ and often refer to the subgroups in the $\mathcal{C}_1$ collection as the \emph{subspace} subgroups since this collection consists of stabilisers of subspaces, or pairs of subspaces of $V$. Conversely, we refer to any subgroup not contained in $\mathcal{C}_1$ as a \emph{non-subspace} subgroup. 

\begin{table}
\[
\begin{array}{ll} \hline
\mathcal{C}_1 & \mbox{Stabilisers of subspaces, or pairs of subspaces of $V$}\\
\mathcal{C}_2 & \mbox{Stabilisers of direct sum decompositions } V=\bigoplus^t_{i=1}V_i \mbox{, where} \dim V_i=a\\
\mathcal{C}_3 & \mbox{Stabilisers of prime degree extension fields of } \mathbb{F}_q\\
\mathcal{C}_4 & \mbox{Stabilisers of tensor product decompositions } V=V_1\otimes V_2\\
\mathcal{C}_5 & \mbox{Stabilisers of prime index subfields of } \mathbb{F}_q\\
\mathcal{C}_6 & \mbox{Normalisers of symplectic-type $r$-groups, } r\neq p\\
\mathcal{C}_7 & \mbox{Stabilisers of tensor product decompositions }V=\bigotimes^t_{i=1}V_i \mbox{, where } \dim V_i=a \\
\mathcal{C}_8 & \mbox{Stabilisers of non-degenerate forms on $V$}\\
\mathcal{S}& \mbox{Almost simple absolutely irreducible subgroups}\\
\hline
\end{array}
\]
\caption{Aschbacher's subgroup collections}
\label{tab:geom}
\end{table}

Set $H_0=H\cap G_0$ and assume that $H\not\in\mathcal{S}$. If $H_0$ is a maximal subgroup of $G_0$ then $H_0\in\mathcal{C}$. However there are a few cases in which $H_0$ is non-maximal: this leads to a small additional subgroup collection when $G_0={\rm PSp}_4(2^f)$ or ${\rm P}\Omega^+_8(q)$, that arises due to the existence of exceptional automorphisms. Following \cite{BG_book}, we denote this subgroup collection $\mathcal{N}$ and refer to the elements of $\mathcal{N}$ as \textit{novelty} subgroups. Note the maximal subgroups in the case when $G_0={\rm P}\Omega^+_8(q)$ were determined up to conjugacy by Kleidman \cite{Klei}. The members of $\mathcal{N}$ are outlined in Table \ref{tab:novelty}.

\begin{table}
\[
\begin{array}{lll} \hline
G_0 & \mbox{Type of } H & \mbox{Conditions}\\
\hline \rule{0pt}{2.5ex}
{\rm PSp}_4(2^f) & {\rm O}^{\epsilon}_2(q)\wr S_2 & \\
                         & {\rm O}^-_2(q^2).2& \\
                         & [q^4].{\rm GL}_1(q)^2& \\
{\rm P}\Omega_8^+(q) & {\rm GL}^{\epsilon}_1(q)\times{\rm GL}^{\epsilon}_3(q)& \\
                         &{\rm O}^-_2(q^2)\times {\rm O}^-_2(q^2) & \\
                         &{\rm G}_2(q) & \\
                         & [2^9].{\rm SL}_3(2)& q=p>2 \\
                         & [q^{11}].{\rm GL}_2(q){\rm GL}_1(q)^2& \\
\hline
\end{array}
\]
\caption{The collection $\mathcal{N}$}
\label{tab:novelty}
\end{table}
The following is a version of Aschbacher's theorem 
\begin{thm}
Let $G$ be an almost simple classical group with socle $G_0$, and let $H$ be a maximal subgroup of $G$ not containing $G_0$. Then $H\in \mathcal{C}\cup\mathcal{N}\cup\mathcal{S}$.
\end{thm}

Following \cite{KL}, we will often refer to the \emph{type} of a maximal subgroup $H$ of $G$. If $H\in\mathcal{S}$ we use the type of $H$ to denote the socle of the almost simple group $H$. Additionally, if $H\in\mathcal{C}\cup\mathcal{N}$ then in most cases the type provides the approximate group-theoretic structure of the subgroup, and moreover often suggests which kind of object is stabilised by $H$. For example, take $G_0=\Li_n(q)$ and $H$ of type ${\rm GL}_a(q)\wr S_t$, then $H$ is the stabiliser of a direct sum decomposition $V=V_1\oplus\dots\oplus V_t$, where $\dim V_i=a$. There are some exceptions to this type of notation, namely we use $P_i$ to denote the stabiliser of a totally singular $i$-space and note that if $G_0=\Li_n(q)$, then all subspaces are totally singular. Additionally, in the case $G_0=\Li_n(q)$ with $n\geqs 3$ we use $P_{i,n-i}$ to denote the stabiliser of a pair of subspaces $U$ and $W$ such that $U<W$ with $\dim(U)=i$ and $\dim(W)=n-i$. We note that subgroups of type $P_{i,n-i}$ are only maximal in the linear case when $G$ contains a graph or a graph-field automorphism. These are another example of a novelty subgroup (since $H_0$ is non-maximal in $G_0$ in this case); however following \cite{KL} we include these subgroups in the $\mathcal{C}_1$ collection.

Kleidman and Liebeck's book, \cite{KL}, is the definitive reference for details on the existence, maximality and structure of the geometric subgroups. In \cite{KL}, they provide a complete description of the structure of the geometric subgroups for all $n$, and determine their maximality (up to conjugacy) for $n\geqs 13$. Additionally, Bray, Holt and Roney-Dougal (\cite{BHR}), recently completely determined all  maximal subgroups (up to conjugacy) of the low-dimensional classical groups with $n\leqs 12$.

\subsection{Conjugacy classes}\label{s:conjclass}
In this subsection we aim to provide a brief description of the conjugacy classes of prime order elements in classical groups, and introduce notation for the representatives of such classes. For the description of these conjugacy classes we follow \cite[Chapter 3]{BG_book}. 
Additionally we follow \cite[Chapter 2]{BG_book} for the definitions of the outer automorphisms of finite simple classical groups. In particular the diagonal, field, graph and graph-field automorphisms. We note that due to a theorem of Steinberg \cite[Theorem 30]{Stein} any element in an almost simple classical group is the product of an inner, a diagonal, a field and a graph automorphism. 
 
Let ${\rm Inndiag}(G_0)$ denote the subgroup of $\Aut(G_0)$ generated by the inner and diagonal automorphisms of $G_0$ (for example ${\rm Inndiag}(\Li_n(q))={\rm PGL}_n(q)$). Take $x\in {\rm Inndiag}(G_0)$ to be an element of prime order $r$. The analysis of conjugacy classes of elements of prime order in classical groups is split into two cases: semisimple and unipotent. We say that $x$ is \emph{semisimple} if $(r,p)=1$ and we say $x$ is \emph{unipotent} if $r=p$. We begin with a discussion of the semisimple elements. 

\subsubsection{Semisimple elements}
For the purposes of this paper we focus our attention on semisimple elements of odd prime order (although we direct the reader to \cite[Chapter 3]{BG_book} for information on semisimple involutions). We begin by stating \cite[Theorem 4.2.2(j)]{GLS}. 

\begin{thm}\label{t:semisimp}
Suppose $x\in {\rm Inndiag}(G_0)$ is a semisimple element of prime order. Then $x^{G_0}=x^{{\rm Inndiag}(G_0)}$.
\end{thm}
In particular, this tells us that for prime order semisimple elements in $G_0$ there are the same number of conjugacy classes in $G_0$ as in ${\rm Inndiag}(G_0)$.

We now discuss the notation we use for semisimple elements. Take $x\in{\rm Inndiag}(G_0)$ to be an element of odd prime order $r\neq p$ such that $r$ is a primitive prime divisor of $q^i-1$ for some $i\geqs 1$. Following \cite{BG_book}, we define 
\begin{equation*}
c=
\begin{cases}
2i & \mbox{if } i \mbox{ is odd and } G_0\neq \Li_n(q)\\
i/2 & \mbox{if } i\equiv 2\imod 4 \mbox{ and } G_0= \Un_n(q)\\
i &\mbox{otherwise}
\end{cases}.
\end{equation*} 

Assume $c\geqs 2$ then by \cite[Lemma 3.1.3]{BG_book} we can write 
$$x=\hat{x}Z,$$
for some unique $\hat{x}$ of order $r$ in the corresponding matrix group to $G_0$ and where $Z$ is the center of this matrix group.
In order to highlight some of the relevant notation we provide the details in the linear case, that is in the case when $G_0=\Li_n(q)$ (note that the notation and set up is similar in the other classical groups with slight variations, see \cite[Chapter 3]{BG_book}). Set $G={\rm GL}_n(q)$ and let $\hat{x}\in G$ be an element of order $r$ and $T(r)$ denote the set of non-trivial $r^{{\rm th}}$ roots of unity in $\mathbb{F}_{q^i}$. Then $\hat{x}$ is diagonalisable over $\mathbb{F}_{q^i}$ but not any proper subfield, so $\hat{x}$ fixes a direct sum decomposition 
$$V=U_1\oplus \dots\oplus U_s\oplus C_V(\hat{x})$$
by Maschke's Theorem. Here $C_V(\hat{x})$ denotes the 1-eigenspace of $\hat{x}$ and each $U_j$ is an $i$-dimensional subspace on which $\hat{x}$ acts irreducibly. The set of eigenvalues of $\hat{x}$ on $U_j\otimes\mathbb{F}_{q^i}$ are of the form $\Lambda=\{\lambda,\lambda^q,\dots,\lambda^{q^{i-1}}\}$ for some $\lambda\in T(r)$; this is precisely an orbit on $T(r)$ under the action of the Frobenius automorphism $\sigma:\mu\longmapsto \mu^q$ of $\mathbb{F}_{q^i}$. There are $t=(r-1)/i$ distinct $\sigma$-orbits, which we label as $\Lambda_1,\dots,\Lambda_t$. By \cite[Lemma 3.1.7]{BG_book}, two elements of order $r$ in $G$ are conjugate if and only if they have the same multiset of eigenvalues in $\mathbb{F}_{q^i}$, so we may abuse notation and write 
$$\hat{x}=[\Lambda_1^{a_1},\dots,\Lambda_t^{a_t},I_e],$$
where $a_j$ denotes the multiplicity of $\Lambda_j$ as a multiset of eigenvalues of $\hat{x}$ on $V\otimes \mathbb{F}_{q^i}$ and $e=\dim C_V(\hat{x})$. In particular, the conjugacy classes of elements of order $r$ in ${\rm PGL}_n(q)={\rm Inndiag}(\Li_n(q))$ are uniquely determined by the multisets of eigenvalues in $\mathbb{F}_{q^i}$. With suitable changes there are similar descriptions for the other classical groups ${\rm GU}_n(q)$, ${\rm Sp}_n(q)$ and ${\rm O}^{\epsilon}_n(q)$. For example, the main difference in ${\rm Sp}_n(q)$ is that when $i$ is odd the $\Lambda_j$ sets arise in inverse pairs. Thus elements of order $r$ in ${\rm PGSp}_n(q)$, when $i$ is odd, have the form $x=[(\Lambda_1,\Lambda_1^{-1})^{a_1},\dots,(\Lambda_{t/2},\Lambda_{t/2}^{-1})^{a_{t/2}},I_e]Z$. We direct the reader to \cite[Chapter 3]{BG_book} for more details.

Now assume $c=1$. In this case we still represent elements of prime order $r$ as $x=\hat{x}Z$, with $\hat{x}$ an element of order $r$ in the corresponding matrix group to $G_0$ and $Z$ is the center of this matrix group. However in this case $x$ does not lift to a unique element $\hat{x}$ of order $r$ in the corresponding matrix group. Here extra conjugacy classes of elements of order $r$ exist that are not present in the case $c\geqs 2$. We note that $c=1$ only occurs when $G_0=\Li_n(q)$ and $i=1$ or $G_0=\Un_n(q)$ and $i=2$. In both of these cases a similar description of conjugacy is available, see \cite[Propostion 3.2.2]{BG_book} and \cite[Propostion 3.3.3]{BG_book} respectively. 

\subsubsection{More on unitary groups}
We now turn our attention to the unitary case. These results will be useful to prove part (i) of Theorem \ref{t:mainthrm} (see Proposition \ref{p:uniprop}). 

Let $G_0=\Un_n(q)$ and suppose that $r$ is a prime divisor of $|G_0|$ such that $r$ is a primitive prime divisor of $q^i-1$ where $i\equiv 2\imod 4$ and $i\geqs 10$. In order to handle case (i) of Theorem \ref{t:mainthrm} we need to know how the non-inner automorphisms of $G_0$ act on the set of conjugacy classes of elements of order $r$ in $G_0$. 
Any element $x\in {\rm PGU}_n(q)$ of order $r$ is in fact an element of $G_0$ since $|{\rm PGU}_n(q){:}G_0|=(n,q+1)$ is indivisible by $r$. Additionally, $x$ can be written as $x=\hat{x}Z$ where $Z=Z({\rm GU}_n(q))$ and $\hat{x}\in{\rm GU}_n(q)$ is of order $r$. By \cite[Proposition 3.3.2]{BG_book} $\hat{x}$ fixes an orthogonal decomposition of $V$ (the natural $G_0$ module) into irreducible blocks and $C_V(\hat{x})$. Here $C_V({\hat{x}})$ is non-degenerate (or trivial) and the irreducible blocks are non-degenerate $i/2$-spaces on which $\hat{x}$ acts irreducibly with eigenvalues 
\begin{equation}\label{e:unilam}
\Lambda_j=\{\lambda_j,\lambda_j^{q^2},\dots,\lambda_j^{q^{2(b-1)}}\},
\end{equation}
for some $r^{{\rm th}}$ root of unity $\lambda_j\in\mathbb{F}_{q^i}$ and where $b=i/2$. Thus we may abuse notation and write $\hat{x}=[\Lambda_1^{a_1},\dots,\Lambda_s^{a_s},I_e]$, where $s=(r-1)/b$, $a_j$ is the multiplicity of $\Lambda_j$ in the multiset of eigenvalues of $\hat{x}$ and $e=\dim C_V(\hat{x})$ (note the $\Lambda_j$ sets coincide with the $\sigma^2$-orbits in $T(r)$, where $\sigma^2:\mu\longmapsto\mu^{q^2}$). From \cite[Proposition 3.3.2]{BG_book} there exists a bijection 
$$\theta:(a_1,\dots,a_s)\longmapsto([\Lambda_1^{a_1},\dots,\Lambda_s^{a_s},I_e]Z)^{{\rm PGU}_n(q)}$$
between the non-zero $s$-tuples $(a_1,\dots,a_s)\in \mathbb{N}_0^s$ such that $i\sum_{j}a_j\leqs 2n$ and the set of ${\rm PGU}_n(q)$ classes of elements of order $r$ in $G_0$. Additionally we note that $x^{G_0}=x^{{\rm PGU}_n(q)}$ by Theorem \ref{t:semisimp} for any element of order $r$ in $G_0$.

Following \cite{KL}, for $x\in \Aut(G_0)$ we use $\ddot{x}$ to denote the coset $G_0x\in {\rm Out}(G_0)=\Aut(G_0)/G_0$. By 
\cite[Proposition 2.3.5]{KL}, 
$${\rm Out}(G_0)=\langle \ddot{\delta}\rangle{:} \langle \ddot{\phi} \rangle,$$ 
where $|\ddot{\delta}|=(n,q+1)$, $|\ddot{\phi}|=2f$ and $\ddot{\delta}^{\ddot{\phi}}=\ddot{\delta}^p$.
With respect to an orthonormal basis $\{v_1,\dots,v_n\}$ for $V$ we may assume that $\phi$ is the field automorphism of order $2f$ corresponding to the Frobenius map $\sum_i \lambda_iv_i\longmapsto \sum_i\lambda_i^pv_i$ on $V$, and $\delta$ is the diagonal automorphism of order $(n,q+1)$ induced by conjugation by $[\mu,I_{n-1}]$, where $\mu\in\mathbb{F}_{q^2}$ has order $q+1$. 

Take $G=G_0.J$ such that $J\leqs {\rm Out}(G_0)$ and define $\Phi=\{\Lambda_1,\dots,\Lambda_s\}$. Since the diagonal automorphisms act trivially on the set $\Phi$, they do not affect the number of conjugacy classes of elements of order $r$ in $G_0$. However the field automorphisms $\phi^l\in\langle \phi \rangle$ act on $\Phi$ as 
$$\phi^l\cdot\Lambda_j=\{\mu_j^{p^l},\mu_j^{q^2p^l},\dots,\mu_j^{q^{2(b-1)}p^l}\}\in\Phi,$$
 so $\langle \phi \rangle$ induces a permutation on $\Phi$. Thus the number of $G$-classes of elements of order $r$ in $G_0$ depends entirely on how $J$ projects onto $\langle \ddot{\phi} \rangle$. In Lemma \ref{l:phiuni} we prove precisely how many orbits the group $\langle \phi^k \rangle$ has with its action on $\Phi$, where $k$ is some divisor of $2f$.

We begin by stating a well known number theoretic result (see \cite[Proposition 3.3.4]{Numthry} for example). 
\begin{lem}\label{l:mod solns}
Suppose that $a,b,c\in\mathbb{Z}$ such that $a,c\neq 0$ and $(a,c)=d$. Then the linear congruence $ax-b\equiv 0\imod c$ has solutions if and only if $d$ divides $b$. Moreover, if $d$ divides $b$ then there exist exactly $d$ solutions modulo $c$. 
\end{lem}

Recall that $P_a^b$ denotes the set of primitive prime divisors of $a^b-1$ for positive integers $a$ and $b$. In the following lemma we use $i$, $\Lambda_j$, $\Phi$, $q$ and $\phi$ as defined above. 

\begin{lem}\label{l:phiuni}
Suppose $r\in P_q^i\cap P_p^m$ for some $m\leqs if$ and let $D=\langle \phi^k \rangle$ where $2f=kh$. Define $a:=(m,f)$ and $d:=(a,k)$. Then the orbits of $D$ acting on $\Phi$ are of size $\frac{at}{k}$, where $t=2$ if $\frac{k}{d}$ is odd, otherwise $t=1$.
\end{lem}
\begin{proof}

Since $a=(m,f)$ we may write $m=av$ and $f=az$ such that $(v,z)=1$. By assumption $r$ is a primitive prime divisor of both $q^i-1$ and $p^m-1$, which implies that $v=i$, that is $m=ai$. We note that $a$ is odd since $(a,i)=1$ and $i$ is even.

In view of the orbit stabiliser theorem we focus our attentions on the size of the stabilisers in $D$ of each $\Lambda_j\in\Phi$, which we denote as $D_{\Lambda_j}$. 
Fix $\Lambda_j\in\Phi$ and take $\phi^{lk}\in D$ for some $0\leqs l<h$. We may assume $l\geqs 0$, otherwise we have the identity element. Here $\phi^{lk}\in D_{\Lambda_j}$ if and only if $\lambda_j^{q^{2w}p^{lk}}=\lambda_j$ for some $1\leqs w\leqs i/2-1$. Since $\lambda_j\in T(r)$, this occurs if and only if $p^{2wf+lk}\equiv 1\imod r$. In turn this occurs if and only if $m$ divides $2fw+lk$, which is equivalent to saying that $lk=xa$ for some $1\leqs x\leqs 2z-1$ and that there exists a $c\in\{1,\dots,2z-1\}$ such that $2zw=ci-x$. 

Assume $lk=xa$ for some $1\leqs l<h$ and $1\leqs x\leqs 2z-1$. First suppose that $x$ is odd. Since $i$ is even there does not exist a $c$ such that $2zw=ci-x$. Next suppose $x$ is even and note that $z$ is odd since $(i,z)=1$. Then by Lemma \ref{l:mod solns} there always exists at least one $1\leqs c \leqs 2z-1$ such that $ci-x\equiv 0\imod {2z}$. So we conclude that $\phi^{lk}\in D_{\Lambda_j}$ for $1\leqs l<h$ if and only if $lk=xa$ for some even $1\leqs x\leqs 2z-1$.

Suppose first that $k/d$ is even. Then $lk/a$ is even, so $|D_{\Lambda_j}|$ is precisely the number of multiples of $a/k$ in $\{0,\dots,h-1\}$. Thus $|D_{\Lambda_j}|=2z$ since $h=2za/k$. 
Finally suppose that $k/d$ is odd. In this case $lk/a$ is even if and only if $l$ is even. Thus in a similar manner $|D_{\Lambda_j}|$ is precisely the number of even multiples of $a/k$ in $\{0,\dots,h-1\}$, so $|D_{\Lambda_j}|=z$.

Since the size of $|D_{\Lambda_j}|$ is independent of $j$ we conclude that all the orbits have the same size. Thus the result follows by the orbit stabiliser theorem since $|D|=h=2za/k$.
\end{proof}

\begin{cor}\label{c:unitary case cor}
Let $G$ be an almost simple group with socle $G_0=\Un_n(q)$, where $n\geqs 5$ is odd. Let $r$ be a primitive prime divisor of both $q^{2n}-1$ and $p^m-1$ and define $a:=(m,f)$. Then there exists a unique $G$-class of elements of order $r$ in $G_0$ if and only if
\begin{itemize}
\item[{\rm (i)}] $r=2na+1$; and 
\item[{\rm (ii)}] $G/G_0$ projects onto $\langle \ddot{\phi}\rangle$.
\end{itemize}
\end{cor}
\begin{proof}
Let $K_{G_0}(G,r)$ be the number of $G$-classes of elements of order $r$ in $G_0$ and let $J = G/G_0 \leqs {\rm Out}(G_0)$. 
Since $r$ is a primitive prime divisor of $p^m-1$, $m=2na$ and $r=2naw+1$ for some positive integer $w$. 
The ${\rm PGU}_n(q)$-classes of elements of order $r$ in $G_0$ are represented by the elements $x_j=[\Lambda_j]Z$, where $\Lambda_j$ is as defined in \eqref{e:unilam} and $1\leqs j\leqs s=(r-1)/n=2aw$ (see \cite[Proposition 3.3.2]{BG_book}).

Assume first that the projection of $J$ to $\langle \ddot{\phi}\rangle$ is trivial. Then $G\leqs {\rm PGU}_n(q)$ so $K_{G_0}(G,r)=2aw\geqs 2$. Now assume that the projection of $J$ to $\langle \ddot{\phi}\rangle$ is nontrivial, say $J$ projects onto $\langle \ddot{\phi}^k\rangle$ for some $1\leqs k<2f$. Then by Lemma \ref{l:phiuni}, $K_{G_0}(G,r)= 2kw/t$ where $t=2$ if $k/(a,k)$ is odd and $t=1$ otherwise. Therefore $K_{G_0}(G,r)= 2kw/t=1$ if and only if $k=w=1$.
\end{proof}

\subsubsection{Unipotent elements}

Now we say a few words on the unipotent classes of elements of order $p$. Let $x\in {\rm Inndiag}{(G_0)}$ be an element of order $p$. Then by \cite[Lemma 3.1.3]{BG_book}, $x$ lifts to an element of order $p$ in the corresponding matrix group, so we may write $x=\hat{x}Z$, where again $\hat{x}$ is an element of order $p$ in the corresponding matrix group to $G_0$ and $Z$ is the center of this matrix group. Then here we may write $\hat{x}$ as the Jordan decomposition of $\hat{x}$ on $V$, 
$$\hat{x}=[J_p^{a_p},\dots,J_1^{a_1}],$$
where $J_i$ denotes a standard unipotent Jordan block of size $i$ and $a_i$ its multiplicity. In particular, in the symplectic and orthogonal cases there are conditions on the multiplicities of the Jordan blocks. A convenient source for the following lemma is \cite[Lemma 3.4.1, 3.5.1]{BG_book}.

\begin{lem}\label{l:osjordan}
Suppose $\hat{x}\in{\rm Sp}_n(q)$ or ${\rm O}^{\epsilon}_n(q)$ is an element of order $p$ odd and with Jordan form $[J_p^{a_p},\dots,J_1^{a_1}]$. If $\hat{x}\in{\rm Sp}_n(q)$ then $a_i$ is even for all odd $i$. Similarly if $\hat{x}\in{\rm O}^{\epsilon}_n(q)$ then $a_i$ is even for all even $i$.
\end{lem}

In the linear case elements of order $p$ are conjugate in ${\rm GL}_n(q)$ if and only if they have the same Jordan form. Again there is a similar description of the conjugacy classes of unipotent elements in other classical groups ${\rm GU}_n(q)$, ${\rm Sp}_n(q)$ and ${\rm O}^{\epsilon}_n(q)$. In the symplectic and orthogonal cases the main difference is in the case $p=2$. Here elements of order $p$ have Jordan form $[J_2^s,J_1^{n-2s}]$ and $G_0$ contains at least two conjugacy classes of elements with this Jordan form when $s$ is even. We refer the reader to \cite[Chapter 3]{BG_book} for more details.

\section{Proof of Theorem \ref{t:primedivs}}\label{s:primedivs}
Let $G$ be an almost simple classical group with socle $G_0\in\widetilde{\mathcal{G}}$, where we recall that $\widetilde{\mathcal{G}}$ is the set of all finite simple classical groups over $\mathbb{F}_q$, not including duplicates obtained via isomorphisms
and $q=p^f$ for $p$ prime and $f\geqs 1$. Let $H$ be a point stabiliser in $G$ and recall that $H_0=H\cap G_0$ and $\pi(X)$ denotes the number of prime divisors of $|X|$. In this section we prove Theorem \ref{t:primedivs}, providing a classification of the pairs $(G_0,H)$ such that $\pi(G_0)-\pi(H_0)\leqs 1$.
Recall that $H\in \mathcal{C}\cup \mathcal{N}\cup \mathcal{S}$ by Aschbacher's subgroup structure theorem (see Section \ref{s:subgrp struc}). We split the proof of Theorem \ref{t:primedivs} into two subsections. Firstly we handle the case when $H\in \mathcal{C}\cup\mathcal{N}$, that is $H$ is a geometric or novelty subgroup. Secondly, we handle the cases in which $H$ is a non-geometric subgroup, that is $H\in\mathcal{S}$. For this section we adopt a different approach in general, since unlike the subgroups in $\mathcal{C}\cup\mathcal{N}$ a complete list of the subgroups contained in $\mathcal{S}$ does not exist. For these subgroups we use results found in Guralnick et al. \cite{GPPS}, in which they provide a description of subgroups $M\in {\rm GL}_n(q)$ such that $|M|$ is divisible by a primitive prime divisor of $q^i-1$ for $\frac{n}{2}<i\leqs n$. This provides a way of finding prime divisors of $|G_0|$ that do not divide $|H_0|$. Additionally, for certain low dimensional cases we use the tables in Bray, Holt and Roney-Dougal \cite{BHR}.

\begin{table}
\[
\begin{array}{llcll} \hline
\mbox{Case} & G_0 & &\mbox{Type of } H & \mbox{Conditions} \\ \hline
\mbox{I} & \Li_n(q) &\mathcal{C}_1 & {\rm P}_1 & (n,q)=(6,2)\\
\mbox{II} &            &                    & {\rm GL}_1(q)\oplus{\rm GL}_{n-1}(q) & (n,q)=(6,2)\\
\mbox{III}&            &\mathcal{S} & A_7 & (n,q)=(4,2)\\
\mbox{IV} & \Un_n(q) &\mathcal{C}_1 & {\rm P}_2 & (n,q)=(4,2)\\
\mbox{V} &              &\mathcal{C}_5 & {\rm Sp}_n(q) & (n,q)=(4,2)\\
\mbox{VI}&          &\mathcal{S} &{\rm M}_{22} & (n,q)=(6,2)\\
\mbox{VII} &             &                 & \Li_{2}(11) & (n,q)=(5,2)\\
\mbox{VIII} &             &                 & \Li_{2}(7) & (n,q)=(3,3)\\
\mbox{IX} &             &                 & \Li_{3}(4) & (n,q)=(4,3)\\
\mbox{X} &             &                 & A_7 & (n,q)=(3,5),(4,3)\\
\mbox{XI} & {\rm PSp}_n(q) & \mathcal{C}_3 & {\rm Sp}_{n/2}(q^2) & n=4 \\
\mbox{XII} &                      & \mathcal{C}_8 & {\rm O}^-_n(q) & n\equiv 0\imod 4,\mbox{ or } (n,q)=(6,2) \\
\mbox{XIII} &                        & \mathcal{S}& A_7 & (n,q)=(4,7)\\
\mbox{XIV} &{\rm P}\Omega^{+}_n(q) & \mathcal{C}_1 & {\rm P}_m & (n,q,m)=(8,2,1),(8,2,4) \\
\mbox{XV} &                                    &                      & {\rm O}_1(q)\perp{\rm O}_{n-1}(q) & n\equiv 0\imod 4 \\
\mbox{XVI} &                                    &                      &{\rm Sp}_{n-2}(q) & n\equiv 0\imod 4\\
\mbox{XVII} &                                      & \mathcal{S} & \Omega_7(q) & n=8 \mbox{ and } p\neq2 \\
\mbox{XVIII} &                                     &                  & {\rm Sp}_6(q) & n=8 \mbox{ and } p=2 \\
\mbox{XIX} &             &                 & A_9& (n,q)=(8,2)\\
\mbox{XX} & \Omega_n(q) & \mathcal{C}_1&{\rm O}_1(q)\perp{\rm O}^-_{n-1}(q) & n\equiv 1\imod 4 \\
\hline
\end{array}
\]
\caption{ $\pi(G_0)=\pi(H_0)$}
\label{tab:pi0}
\end{table}

\begin{table}
\[
\begin{array}{lllll} \hline
 G_0 & (n,q,r)& \mbox{Type of } H\\ \hline
\Li_n(q) & (7,2,127) & {\rm P}_2,\, {\rm GL}_2(q)\oplus{\rm GL}_{5}(q),\,{\rm P}_{1,6}\\
           & (6,2,31)  & {\rm P}_2,\, {\rm GL}_2(q)\oplus{\rm GL}_{4}(q),\,{\rm P}_{1,5}\\
           & (5,3,13)  & {\rm M}_{11} \\
           & (4,4,17) & {\rm GL}_4(q^{1/2})\\
           & (3,8,73)&{\rm GL}_1(q)\wr S_3,\, {\rm GL}_3(q^{1/3})\\
           & (3,4,7)&  A_6 \\
           & (3,4,5) & {\rm GL}_3(q^{1/2})\\
           & (2,2^k-1,2) & {\rm P}_1 \\
           & (2,2^k-1,2^k-1)&{\rm GL}_1(q)\wr S_2 \\
           & (2,2^k+1,2^k+1)& {\rm GL}_{1}(q^2) \\
           & (2,2^f,2^f-1)&  {\rm GL}_{1}(q^2)\\
           & (2,p,p)^{\dagger}&  2^2.{\rm Sp}_2(2) \\
           & (2,q,p)^{\dagger\dagger} & A_5 \\
          & (2,9,3)&  {\rm GL}_{1}(q^2) \\
 \Un_n(q) & (6,2,11)&\Un_4(3) \\
              & (5,2,11) &{\rm P}_2 ,\,{\rm GU}_1(q)\wr S_5 \\
              & (4,5,13) &\Un_4(2)  ,\,A_7 \\
              & (4,3,7) &2^{4}.{\rm Sp}_{4}(2) \\
              & (4,2,5) &{\rm P}_1 ,\,{\rm GU}_1(q)\wr S_4  \\
              & (3,9,73) & {\rm GU}_1(q)\wr S_3\\
              & (3,5,7) &A_6 \\
              & (3,5,5)&\Li_2(7)  \\
              & (3,4,13) &{\rm GU}_1(q)\wr S_3  \\
              & (3,3,7) &{\rm GU}_1(q)\wr S_3 \\
 {\rm PSp}_n(q) & (8,2,17) &{\rm P}_1 ,\, {\rm P}_4 ,\,{\rm Sp}_2(q)\perp{\rm Sp}_{6}(q),\, A_{10}\\
                      & (8,2,7)&{\rm Sp}_{4}(q^2) \\
                      & (6,3,5) &\Li_2(13) \\
                      & (6,2,7) &{\rm P}_1 ,\, {\rm Sp}_2(q)\perp{\rm Sp}_{4}(q) \\
                      & (6,2,5)&{\rm P}_3 \\
                      & (4,8,3) &{\rm Sz}(q) \\
                      & (4,7,7)& 2^{1+4}.{\rm O}^-_{4}(2)\\
                      & (4,5,13) &2^{1+4}.{\rm O}^-_{4}(2),\,  A_6\\
{\rm P}\Omega^{+}_n(q) & (10,2,17) & {\rm P}_5 ,\,  {\rm GL}_{5}(q).2\\
                                    &(8,3,13) &{\rm O}_1(q)\wr S_8,\, 2_+^{1+6}.{\rm O}^+_{6}(2) ,\, \Omega^+_8(2) \\
                                    &(8,2,7) &{\rm O}^-_{4}(q)\wr S_2,\, {\rm O}^{+}_{4}(q^2)\\
                                    &(8,2,5) &{\rm P}_3 ,\, {\rm GL}_1(q)\times{\rm GL}_3(q)\\

{\rm P}\Omega^{-}_{n}(q) & (10,2,17) &A_{10} ,\,{\rm M}_{12}\\
                                      & (8,2,17) & {\rm O}^-_2(2)\perp{\rm O}^+_6(2) \\
                                      & (8,2,7) &{\rm O}^{-}_{4}(q^2)  \\

\Omega_n(q) & (7,3,13) & {\rm O}_1(q)\wr S_7 ,\, {\rm Sp}_6(2) ,\,A_9  \\
\hline
\end{array}
\]
\caption{Cases with $\pi(G_0)=\pi(H_0)+1$: Part I}
\label{tab:pi2}

\end{table}

\begin{table}
\[
\begin{array}{llllll} \hline
 \mbox{Case} & G_0 & &\mbox{Type of } H & \mbox{Conditions} & i\\ \hline
\mbox{L1} & \Li_n(q) &\mathcal{C}_1 & {\rm P}_1 && n\\
\mbox{L2} &            &                    &{\rm GL}_1(q)\oplus {\rm GL}_{n-1}(q) &&n \\
\mbox{L3} &            &                    &{\rm P}_{1,n-1} & n=3, q=p \mbox{ Mersenne} & n\\
\mbox{L4} &            &\mathcal{C}_2 & {\rm GL}_{1}(q)\wr S_n & (n,p)=(2,2) & n\\
\mbox{L5} &            &\mathcal{C}_3 & {\rm GL}_{n/2}(q^2) & n=4,6& n-1\\
\mbox{L6} &            &\mathcal{C}_5 & {\rm GL}_{n}(q^{1/2}) & n=2 & n\\
\mbox{L7} &            &\mathcal{C}_8 &{\rm Sp}_n(q) & n=4,6& n-1 \\
\mbox{L8} &            &                    &{\rm O}^{\epsilon}_n(q) & (\epsilon,n)=(\circ, 3),(-,4) &3 \\
\mbox{U1} & \Un_n(q) &\mathcal{C}_1 & {\rm P}_{n/2} & n=4,6 &2n-2 \\
\mbox{U2} &              &                    & {\rm P}_1 & n=3  & 2n \\
\mbox{U3} &              &                    &{\rm GU}_1(q)\perp{\rm GU}_{n-1}(q) & & \mbox{See Remark \ref{r:pi tables}(e)} \\
\mbox{U4} &              &\mathcal{C}_2 &{\rm GL}_{n/2}(q^2).2 &n=4,6& 2n-2 \\ 
\mbox{U5} &              &\mathcal{C}_5 &{\rm Sp}_n(q) & n=4,6& 2n-2\\
\mbox{U6} &              &                    & {\rm O}_n^{-}(q) & n=4& 2n-2\\
\mbox{U7} &              &                    & {\rm O}_n(q) & n=3& 2n\\
\mbox{S1} & {\rm PSp}_n(q) &\mathcal{C}_1 & {\rm P}_1 & n\equiv 0 \imod 4& n\\
\mbox{S2} &                      &                    & {\rm P}_2 & n=4 & n\\
\mbox{S3} &                      &                    & {\rm Sp}_2(q)\perp{\rm Sp}_{n-2}(q) & n\equiv 0\imod 4 & n \\
\mbox{S4} &                      &\mathcal{C}_2 & {\rm GL}_{n/2}(q).2& n=4& n\\
\mbox{S5} &                      &                    &{\rm Sp}_{n/2}(q)\wr S_2 & n=4& n\\
\mbox{S6} &                      &\mathcal{C}_3 &{\rm Sp}_{n/3}(q^3) &n=6 & n-2\\
\mbox{S7} &                      &\mathcal{C}_5 & {\rm Sp}_n(q^{1/2}) & n=4& n \\
\mbox{S8} &                      &\mathcal{C}_8 & {\rm O}^{+}_n(q) & & n\\
\mbox{S9} &                      &                    & {\rm O}^{-}_n(q) & n\equiv 2\imod 4 &n/2 \\
\mbox{S10} &                    &\mathcal{S}     & {\rm G}_2(q) & n=6 & n-2\\
\mbox{S11} &                    &                     & \Li_2(q) & n=4 & n\\
\mbox{O}1 &{\rm P}\Omega^{+}_{n}(q) &\mathcal{C}_1 &{\rm P}_1 &n\equiv 0\imod 4& n-2\\
\mbox{O2} &                                        &                    & {\rm P}_4 & n=8 &n-2\\
\mbox{O3} &                                        &                    &{\rm Sp}_{n-2}(q) &  n\equiv 2\imod 4 & n/2\\
\mbox{O4} &                                        &                    & {\rm O}_1(q)\perp{\rm O}_{n-1}(q) & n\equiv 2\imod 4 & n/2\\
\mbox{O5} &                                        &                    &{\rm O}^{+}_2(q)\perp{\rm O}^{+}_{n-2}(q)& n\equiv 0\imod 4 & n-2\\
\mbox{O6}&                                         &                    &{\rm O}_2^-(q)\perp{\rm O}^-_{n-2}(q) & n\equiv 0\imod 4& (n-2)/2\\
\mbox{O7} &                                        &                    & {\rm O}^{-}_2(q)\perp{\rm O}^{-}_{n-2}(q) & n\equiv 2\imod 4 & n/2\\

\mbox{O8} &                                        &\mathcal{C}_2 & {\rm GL}_{n/2}(q).2 & n=8& n-2\\
\mbox{O9} &                                        &\mathcal{C}_3 & {\rm GU}_{n/2}(q) & n=8& (n-2)/2\\
\mbox{O10} &                                      &\mathcal{C}_5 & {\rm O}^{-}_n(q^{1/2}) & n=8 & n-2 \\
\mbox{O11} &                                      &\mathcal{N}    & G_2(q) & n=8 & n/2\\
\mbox{O12} & {\rm P}\Omega^{-}_n(q)  & \mathcal{C}_1&{\rm P}_1 &   n\equiv 2\imod 4& n\\
\mbox{O13} &                                     &                    &{\rm Sp}_{n-2}(q) & & n  \\
\mbox{O14} &                                     &                    &{\rm O}_1(q)\perp{\rm O}_{n-1}(q) & &n \\
\mbox{O15} &                                     &                    &{\rm O}^{+}_2(q)\perp{\rm O}^{-}_{n-2}(q)&n\equiv 2 \imod 4  &n \\
\mbox{O16} &  \Omega_n(q)  &\mathcal{C}_1 &{\rm P}_1 & n\equiv 1\imod 4 & n-1\\
\mbox{O17} &                      &                    &{\rm O}_1(q)\perp{\rm O}^{+}_{n-1}(q) & &n-1\\
\mbox{O18} &                      &                    &{\rm O}_1(q)\perp{\rm O}^{-}_{n-1}(q) &  n\equiv 3\imod 4 \mbox &(n-1)/2\\
\mbox{O19} &                      &                    &{\rm O}^{\mu}_{2}(q)\perp{\rm O}_{n-2}(q)& n\equiv 1\imod 4&n-1\\
\mbox{O20} &                      &\mathcal{S}    & {\rm G}_2(q) & n=7& n-3\\
\hline
\end{array}
\]
\caption{ Cases with $\pi(G_0)=\pi(H_0)+1$:Part II}
\label{tab:pi1}
\end{table}

\begin{rmk}\label{r:pi tables}
Here we provide some remarks on Tables \ref{tab:pi0},\ref{tab:pi2} and \ref{tab:pi1};
\begin{itemize}\addtolength{\itemsep}{0.2\baselineskip}
\item[{\rm (a)}] The conditions presented for the cases in Tables \ref{tab:pi0} and \ref{tab:pi1} are in addition to the conditions given for existence and maximality in \cite[Tables 3.5.A-F]{KL} and \cite[Section 8.2]{BHR}.

\item[{\rm (b)}] The cases recorded in Table \ref{tab:pi2} are specific cases in which $\pi(G_0)-\pi(H_0)=1$ and either $(G_0,H)$ does not appear in Table \ref{tab:pi1}, or it appears in Table \ref{tab:pi1} but there does not exist a primitive prime divisor of $q^i-1$.

\item[{\rm (c)}] In Table \ref{tab:pi2} the second column is labeled $(n,q,r)$. Here $r$ indicates the unique prime that divides $|G_0|$ but not $|H_0|$.

\item[{\rm (d)}]For cases L6, S7 and O10 in Table \ref{tab:pi1}, we specifically require that there is a unique primitive prime divisor of $(q^{1/2})^{2i}-1$. For $i$ even this is equivalent to there being a unique primitive prime divisor of $q^i-1$ (see Lemma \ref{l:prime fac of f}).
\item[{\rm (e)}] In Table \ref{tab:pi2} the $\dagger$ indicates that $q^2-1=2^a3^b$ for some $a,b\geqs 0$ and $\dagger\dagger$ indicates that $q^2-1=2^a3^b5^c$ for some $a,b,c\geqs 0$ and $q\neq 9$.

\item[{\rm (f)}] In case U3 of Table \ref{tab:pi1} we have 
\[
i:=
\begin{cases}
n & n\equiv 0\imod 4\\
n/2 & n\equiv 2\imod 4\\
2n &\mbox{otherwise}
\end{cases}.
\]
\end{itemize}
\end{rmk}

Before we begin the proof we state the following result that will be useful for both geometric and non-geometric subgroups.

\begin{lem}\label{l:r>n+2}
If $n\geqs 7$, then either $(n,q)\in\{(10,2), (9,2),(8,3), (8,2),(7,3),(7,2)\}$ , or there exist distinct prime divisors $r,s> n+2$ of
\begin{equation}\label{e:prod1}
N:=\prod_{i=1}^{m}(q^{2i}-1);
\end{equation}
where $m=\lceil \frac{n-2}{2} \rceil$.
\end{lem}
\begin{proof}
Suppose $P_q^i\neq\emptyset$, then we let $r_i$ denote the largest primitive prime divisor of $q^i-1$. Recall that $r_i=ik_i+1$ for some $k_i\geqs 1$ (see Lemma \ref{l:Lemma A.1}).
 
Assume first that $n\geqs 25$ and let $A=\{j\mid n-12\leqs j\leqs n-1 \mbox { and } j \mbox{ is even}\}$. Take $B=\{r_i\mid i\in A\}$, then every element of $B$ is an odd prime divisor of \eqref{e:prod1}. Suppose first that $k_i=1$ for all $i\in A$. Then $B$ is a set of six consecutive odd numbers all greater than 3, so at least two are not prime, which is a contradiction. Now suppose that $k_i\geqs2$ for exactly one $i\in A$. Then $B$ contains at least three consecutive odd numbers all greater than 3, implying that not all elements of $B$ are prime, which is again a contradiction. Thus $k_i\geqs 2$ for at least two $i\in A$, that is $r_i\geqs 2i+1$. Therefore the lemma holds for $n\geqs 25$, since $2i+1>n+2$.

Now assume $11\leqs n \leqs 24$ and $q\neq 2$. Note that $r_8\geqs 41>n+2$ by Lemma \ref{l:ppds4}, so for this case it remains to find an additional prime divisor of \eqref{e:prod1} larger than $n+2$. For $16\leqs n\leqs 24$ we can take $r_{14}\geqs 29$. By Lemma \ref{l:ppd2a3} we know $r_{12}>25$, so for $n=15,14$ or 13 we take $r_{12}$. Finally if $n=11$ or $12$ then for $q\not\in\{3,5,7,239\}$ Lemma \ref{l:ppds4} implies that $r_4\geqs 17$, and it is straightforward to calculate that for $q\in\{3,5,7,239\}$, $r_{10}\geqs61$. Thus the lemma holds in this case.

Next assume $7\leqs n\leqs 10$ and $q\neq 2$. By Lemma \ref{l:ppds4} and Lemma \ref{l:ppd2a3}, if $q\not\in\{3,5,7,19\}$ then $r_4,r_6\geqs 13>n+2$. The cases $q\in\{3,5,7,19\}$ can easily be handled by direct calculation.

Finally suppose that $q=2$ and $7\leqs n\leqs 24$, this can be done by a straightforward direct calculation. In particular, if $15\leqs n\leqs 24$ then $r_7=127$ and $r_5=31$ (these divide $q^{14}-1$ and $q^{10}-1$ respectively). If $11\leqs n\leqs 14$ then $r_8=17$ and $r_5=31$. For $n=9,10$ the only prime divisor of \eqref{e:prod1} larger than $n+2$ is 17. Additionally if $n=7,8$ there are no prime divisors of \eqref{e:prod1} larger than $n+2$ . 
\end{proof}

\begin{cor}\label{c:r>n+2}
Let $G_0$ be a finite simple classical group over $\mathbb{F}_q$. Let $n$ be the dimension of the natural module of $G_0$ and assume $n\geqs 7$. Then $|G_0|$ is not divisible by distinct primes $r,s> n+2$ only if $(n,q)=(10,2), (9,2), (8,3), (8,2),(7,3),(7,2)$.
\end{cor}
\begin{proof}
By \cite[Table 5.1.A]{KL} when considering prime divisors of $|G_0|$ it suffices to consider the prime divisors of 
\begin{equation*}
N:=\prod_{i=1}^{m}(q^{2i}-1);
\end{equation*}
where $m=\lceil \frac{n-2}{2} \rceil$. Thus the result is clear by Lemma \ref{l:r>n+2}.
\end{proof}

\vs

\subsection{Geometric subgroups} 

Here we prove Theorem \ref{t:primedivs} for $H\in\mathcal{C}\cup \mathcal{N}$. Recall that $\mathcal{C}=\mathcal{C}_1\cup \dots\cup\mathcal{C}_8$ denotes the geometric subgroups, and $\mathcal{N}$ denotes the collection of novelty subgroups in Table \ref{tab:novelty}. These subgroup collections are discussed in Section \ref{s:subgrp struc}. We begin the section by stating a useful result for when $G_0=\Un_n(q)$. 

\begin{lem}
\label{l:m<n-1}
Let $m$ and $n$ be positive integers such that $n\geqs 2$ and $m\leqs n-1$ with equality possible only if $n$ is even. Suppose $r$ is a primitive prime divisor of $q^i-1$ where $i:=a\lfloor \frac{n}{2} \rfloor$ with
\begin{equation}
\label{eq:m<n-1}
a=
\begin{cases}
2 &  n\equiv 0,1\imod 4\\
1  &  n\equiv 2,3 \imod 4\\
\end{cases}.
\end{equation}
Then $r$ does not divide $|\Un_m(q)|$. 
\end{lem}
\begin{proof}
The proof is an easy application of Lemma \ref{l:Lemma A.1}.
\end{proof}

\begin{prop}
Theorem \ref{t:primedivs} holds if $H\in\mathcal{C}\cup \mathcal{N}$.
\end{prop}
\begin{proof}
This may be shown by inspection of the orders of $G_0$ (see \cite[Table 5.1.A]{KL}) and of $H_0$ (see \cite{KL} and \cite{BHR}). In general we search for primitive prime divisors of integers of the form $q^i-1$ that divide $|G_0|$ but not $|H_0|$. The approach of the proof is similar in most cases, so we only provide details in a handful of cases:
\begin{itemize}
\item[{\rm (a)}] $G_0={\rm P}\Omega^+_n(q)$ and $H$ is of type ${\rm O}^-_m(q)\perp{\rm O}^-_{n-m}(q)$ with $2\leqs m<\frac{n}{2}$ even.
\smallskip

\item[{\rm (b)}]$G_0=\Un_n(q)$ and $H$ of type ${\rm P}_m$ with $1\leqs m \leqs n/2$.
\smallskip

\item[{\rm (c)}]$G_0={\Li}_n(q)$ and $H$ is of type ${\rm GL}_{m}(q^k)$ where $n=mk$ and $k$ is prime.
\smallskip

\item[{\rm (d)}]$G_0=\Li_n(q),\Un_n(q),{\rm PSp}_n(q)$ or ${\rm P}\Omega^{+}_n(q)$ and $H$ is a $\mathcal{C}_6$ subgroup.
\smallskip

\item[{\rm (e)}]$G_0={\rm P}\Omega^{\epsilon}_n(q)$ and $H$ is of type ${\rm O}_{1}(q)\wr S_n$.
\end{itemize}

\vs
\vs

\noindent\textbf{Case (a):} \textit{ $G_0={\rm P}\Omega^+_n(q)$, $H$ is of type ${\rm O}^-_m(q)\perp{\rm O}^-_{n-m}(q)$ with $2\leqs m<\frac{n}{2}$ even.}
\vs
\vs

From \cite[Proposition 4.1.6]{KL} all prime divisors of $|H_0|$ divide 
\begin{equation*}
A:=q(q^{\frac{n-m}{2}}+1)\prod_{i=1}^{\frac{n-m-2}{2}}(q^{2i}-1)
\end{equation*}
(to see this note that $n-m>m$). Thus in particular, any primitive prime divisor of $q^j-1$ such that $j>n-m$ is not a prime divisor of $|H_0|$.

Assume first that $m\geqs 6$ (note this implies that $n\geqs 14$, so all primitive prime divisors we take in this case exist). Since $n-m\leqs n-6$ any primitive prime divisor of $q^{n-2}-1$ or of $q^{n-4}-1$ divides $|G_0|$ but not $|H_0|$, so $\pi(G_0)-\pi(H_0)\geqs 2$. 

Now assume $m=4$, as before any primitive prime divisor of $q^{n-2}-1$ does not divide $|H_0|$. If $n\equiv 2 \imod 4$ then by Lemma \ref{l:Lemma A.1} any primitive prime divisor of $q^{\frac{n}{2}}-1$ is not a divisor of $|H_0|$, since $\frac{n}{2}$ is odd and $n-m<n$. By similar reasoning if $n\equiv 0\imod 4$ any primitive prime divisor of $q^{((n-2)/2)}-1$ does not divide $|H_0|$. Thus $\pi(G_0)-\pi(H_0)\geqs 2$.

Finally assume $m=2$. By the same reasoning as above any primitive prime divisors of $q^{\frac{n}{2}}-1$ when $n\equiv 2\imod 4$ or $q^{((n-2)/2)}-1$ when $n\equiv 0\imod 4$ divide $|G_0|$ but not $|H_0|$. These are the only possible primes that divide $|G_0|$ that do not divide $|H_0|$. Thus the result follows.  
\vs 
\vs 

\noindent\textbf{Case (b):} \textit{ $G_0=\Un_n(q)$, $H$ of type ${\rm P}_m$ with $1\leqs m \leqs n/2$}.
\vs
\vs

Here 
\begin{equation*}
\begin{split}
|H_0|& =dq^{m(2n-m)}(q^2-1)|{\rm SL}_m(q^2)||{\rm SU}_{n-2m}(q)|\\
        & =dq^b\prod_{i=1}^m(q^{2i}-1)\prod_{i=2}^{n-2m}(q^i-(-1)^i)
\end{split}
\end{equation*}
where $d=1/(q+1,n)$ and $b=n(n-1)/2$ (see \cite[Proposition 4.1.18]{KL}).

We first assume that $(n,m)\neq(3,1),(4,2),(6,3)$ and $(n,q)\neq (4,2),(5,2),(6,2)$. These assumptions ensure the existence of the primitive prime divisors taken in the main argument. 
Take $r$ and $s$ to be primitive prime divisors of $q^i-1$ and $q^j-1$ respectively, where 
\begin{equation*}
i  = \begin{cases}
2n-2 & n \mbox{ is even}\\
2n & n\mbox{ is odd}
\end{cases} \\ 
,
\;\;j  =  \begin{cases}
 2n-6 & m\neq1 \mbox{ and } n\mbox{ is even}\\
 2n-4 & m\neq1 \mbox{ and } n\mbox{ is odd}\\
 a\lfloor n/2\rfloor & m=1
 \end{cases}
\end{equation*}
and $a$ is as defined in \eqref{eq:m<n-1}. By inspection of $|G_0|$ it is easy to see that both $r$ and $s$ are prime divisors of $|G_0|$. For example, since $i/2\leqs n$ is odd, we know that $(q^{i/2}+1)$ divides $|G_0|$, and by definition $r$ is a prime divisor of $q^{i/2}+1$. However, since $2m,2(n-2m)<i$ the definition of a primitive prime divisor implies that $r$ cannot divide $|H_0|$. Using a similar argument for $m\neq 1$ we see that $s$ is not a divisor of $|H_0|$. Finally if $m=1$, then $|H_0|=dq^b(q^2-1)|\Un_{n-2}(q)|$ and so Lemma \ref{l:m<n-1} implies that $s$ does not divide $|H_0|$. Thus $\pi(G_0)-\pi(H_0)\geqs 2$. 

Assume that $(n,m)=(3,1)$ and take $r$ to be a prime divisor of $|G_0|=dq^3(q^2-1)(q^3+1)$ that does not divide $|H_0|=dq^3(q^2-1)$. Then $r$ must be odd since $|H_0|$ is even and must be a prime divisor of $q^3+1$. Thus $r$ must be a primitive prime divisor of either $q^6-1$ or $q^2-1$ by Lemma \ref{l:Lemma A.1}. However $r$ cannot be a prime divisor of $q^2-1$, so the only possible prime divisors of $|G_0|$ that do not divide $|H_0|$ are the primitive prime divisors of $q^6-1$. Thus the result holds in this case. The cases with $(n,m)=(4,2)$ and $(6,3)$ are similar; here the only possible prime divisors of $|G_0|$ that do not divide $|H_0|$ are primitive prime divisors of $q^{2n-2}-1$.

The final cases to handle are those in which $(n,q)=(4,2),(5,2)$ or $(6,2)$. These can be handled by direct calculation of $|G_0|$ and $|H_0|$. Here the only cases with $\pi(G_0)-\pi(H_0)\leqs 1$ are $(n,q,m)=(4,2,1),(4,2,2),(5,2,2),(6,2,3)$ (These are recorded in Table \ref{tab:pi2} and case U1 in Table \ref{tab:pi1}).
\vs
\vs

\noindent\textbf{Case (c):} \textit{$G_0={\Li}_n(q)$, $H$ is of type ${\rm GL}_{m}(q^k)$ where $n=mk$ and $k$ is prime.}
\vs
\vs

Here all prime divisors of $|H_0|$ divide $$A=kq^{\frac{n(m-1)}{2}}\prod_{i=1}^m(q^{ki}-1)$$ by \cite[Proposition 4.3.6]{KL}.

Assume first that $k\neq n$ or $2$ (note this implies $n\geqs 6$). Take $r$ and $s$ to be primitive prime divisors of $q^{n-1}-1$ and $q^{n-2}-1$ respectively. Then $r\geqs n$ and $s\geqs n-1$ by Lemma \ref{l:Lemma A.1}, so $r,s>k$. Additionally we note that $n-2>k(m-1)$ so $r$ and $s$ do not divide $\prod_{i=1}^{m-1}(q^{ki}-1)$. Similarly both $n-1$ and $n-2$ do not divide $n=km$, therefore we conclude that $r$ and $s$ are distinct prime divisors of $|G_0|$ that do not divide $|H_0|$. 

Next assume that $n\geqs 7$ and $k=n$ or $2$. By using a similar argument to the one seen above we see that primitive prime divisors of $q^{n-3}-1$ and $q^i-1$ , where $i=n-1$ if $k=2$ and $n-2$ otherwise, are divisors of $|G_0|$ but not $|H_0|$.
Thus it remains to deal with the cases $(n,k)=(6,2),(5,5),(4,2),(3,3)$ and $(2,2)$.

Assume $(n,k)=(4,2)$ or $(6,2)$. Here the only possible prime divisors of $|G_0|$ that do not divide $|H_0|$ are primitive prime divisors of $q^{n-1}-1$. For example, if $(n,k)=(4,2)$ then $A=2q^2(q^2-1)(q^4-1)$ and $|G_0|=dq^6(q^2-1)(q^3-1)(q^4-1)$ where $d=1/(q-1,4)$. Thus $\pi(G_0)-\pi(H_0)$ is precisely the number of primitive prime divisors of $q^{n-1}-1$, so the result follows.

Next suppose that $(n,k)=(5,5)$. Here $A=5(q^5-1)$ so any primitive prime divisor of $q^3-1$ divides $|G_0|$ and not $|H_0|$. If $p\neq 5$ then $p$ does not divide $|H_0|$ implying that $\pi(G_0)-\pi(H_0)\geqs 2$, so we may assume $p=5$. Take $s$ to be the largest primitive prime divisor of $q^4-1$. Then $s$ divides $|H_0|$ if and only if $s=5$, which by Lemma \ref{l:ppds4} occurs if and only if $q=5$. The final case $q=5$ can be handled by direct calculation. 

Now suppose that $(n,k)=(3,3)$ then 
$$|H_0|=\frac{3(q^3-1)}{(q-1)(q-1,3)}.$$
Immediately we note that if $p\neq 3$ then $p$ does not divide $|H_0|$. Additionally, by Lemma \ref{l:Lemma A.4}, if $r\neq 3$ is a prime divisor of $q-1$, then $r$ is a prime divisor of $|G_0|$ that does not divide $|H_0|$. 
Suppose first that $p\neq 3$. Then since $p$ does not divide $|H_0|$ we may assume $q-1=3^l$ for some $l\geqs 1$ (otherwise $\pi(G_0)-\pi(H_0)\geqs 2$). By Lemma \ref{l:btv} this occurs if and only if $q=4$. It is a simple calculation to show that 2 and 5 divide $|G_0|$ but not $|H_0|$ when $q=4$.
Now suppose that $p=3$. Then 3 does not divide $q-1$ so we may assume that $q-1=r^l$, which by Lemma \ref{l:btv} occurs if and only if $q=9$ or 3. The case $q=3$ does not occur since $H$ is not maximal (see \cite[Table 8.3]{BHR}), and $q=9$ can be handled by direct calculation showing that both 2 and 7 divide $|G_0|$ but not $|H_0|$. 

Finally suppose that $(n,k)=(2,2)$. Here $|H_0|=2.(q+1)$, so we see immediately that the only prime divisors of $|G_0|$ that do not divide $|H_0|$ are $p$ if $p\geqs 3$ and any odd prime divisor of $q-1$. Thus we may assume that either $p=2$ and $q-1=r^l$ for some prime $r$ and $l\geqs 1$, or that $p\geqs 3$ and $q-1=2^l$ for some $l\geqs 1$ (note that in either case $\pi(G_0)-\pi(H_0)=1$). First assume $p=2$ and $q-1=r^l$, by Lemma \ref{l:btv} this happens if and only if $q-1=2^f-1$ is a Mersenne prime. Next assume that $p\geqs 3$ and $q-1=2^l$ then by Lemma \ref{l:btv} this happens if and only if $q=9$ or $q$ is a Fermat prime. Thus the result follows. 

\vs
\vs

\noindent\textbf{Case (d):} \textit{$G_0=\Li_n(q),\Un_n(q),{\rm PSp}_n(q)$ or ${\rm P}\Omega^{+}_n(q)$ and $H$ is a $\mathcal{C}_6$ subgroup.}
\vs
\vs

Here $n=r^m$ with $r$ prime and $p\neq r$. From \cite[Proposition 4.6.5-9]{KL} all prime divisors of $|H_0|$ divide 
$$A:=r\prod_{i=1}^{m}(r^i+1)(r^i-1).$$
Thus if $s$ is a prime divisor of $|H_0|$ then $s\leqs r^m+1=n+1$. In this proof we will use $s_i$ to denote the largest primitive prime divisor of $q^i-1$. 

Suppose first that $n\geqs 7$. By Lemma \ref{l:r>n+2} we easily reduce to the cases $(n,q)=(9,2),(8,3),(8,2),(7,3)$ and $(7,2)$, which can be handled by directly computing $|G_0|$ and $|H_0|$. For example, if $(n,q)=(7,2)$ then $|G_0|=2^{21}.3^4.5.7^2.31.127$ and $|H_0|=7.|{\rm Sp}_2(7)|=2^4.3.7$ implying that $\pi(G_0)-\pi(H_0)\geqs 2$. 

Next suppose $n=5$, then $G_0=\Li_5(q)$ or $\Un_5(q)$ and $A=2^3.3.5$. By Lemma \ref{l:Lemma A.1} we have $s_5,s_{10}\geqs 11$. Additionally $s_4\geqs 13$ when $q\not\in\{2,3,7\}$ by Lemma \ref{l:ppds4}. Note that if $q\in\{2,3,7\}$ then $H$ is not maximal (see \cite[Tables 8.18 and 8.20]{BHR}), thus we do not need to consider these cases. Therefore we conclude that $\pi(G_0)-\pi(H_0)\geqs 2$.

Now suppose $n=4$, then $A=2.3^2.5$. Assume $p\geqs 11$, then $p$ does not divide $|H_0|$ and $s_4\geqs 13$ by Lemma \ref{l:ppds4}, so $\pi(G_0)-\pi(H_0)\geqs 2$. Now assume $p\leqs 7$, then we reduce to the cases $G_0=\Li_4(5),\Un_4(3),\Un_4(7),{\rm PSp}_4(3),{\rm PSp}_4(5)$ and ${\rm PSp}_4(7)$, see \cite{BHR}, (recall that ${\rm PSp}_4(3)\not\in \widetilde{\mathcal{G}}$), which all may be handled by direct calculation. For example $|{\rm PSp}_4(5)|=2^6.3^2.5^4.13$ and in this case $|H_0|=2^6.3.5$, so here $\pi(G_0)-\pi(H_0)=1$.

Suppose $n=3$ then the only prime divisors of $|H_0|$ are 2 and 3. For maximality of $H$, we must have that $p>3$ (see \cite{BHR}), so $p$ does not divide $|H_0|$. Additionally $s_3,s_6\geqs 7$ by Lemma \ref{l:Lemma A.1}, so $\pi(G_0)-\pi(H_0)\geqs 2$.

Finally suppose $n=2$. Then $G_0=\Li_2(q)$ with $q=p\geqs 5$, and the only prime divisors of $|H_0|$ are 2 and 3. Therefore $p$ does not divide $|H_0|$. Additionally the only other possible divisors of $|G_0|$ that do not divide $|H_0|$ are divisors of $q^2-1$ greater than 3. Thus if $q^2-1=2^a.3^b$ for some $a,b\geqs 0$ then $\pi(G_0)-\pi(H_0)=1$, otherwise we have $\pi(G_0)-\pi(H_0)\geqs2$.  

\vs
\vs

\noindent\textbf{Case (e):} \textit{ $G_0={\rm P}\Omega^{\epsilon}_n(q)$, $H$ is of type ${\rm O}_{1}(q)\wr S_n$.}
\vs
\vs

Here $q=p\geqs 3$ and by \cite[Proposition 4.2.15]{KL} all prime divisors of $|H_0|$ divide $n!$. Thus using Lemma \ref{l:r>n+2} we immediately reduce to the cases $(n,q)=(8,3)$ and $(7,3)$. These remaining cases can be handled by direct calculation showing that $\pi(G_0)-\pi(H_0)\leqs 1$ only if $(\epsilon,n,q)=(+,8,3),(\circ,7,3)$. In particular, $\pi(G_0)-\pi(H_0)= 1$ in both cases. 
\end{proof}

\subsection{Non-geometric subgroups}${}$
\vs 

In this subsection we deal with the subgroups contained in the $\mathcal{S}$ collection (see Section \ref{s:subgrp struc} for more details). For $H\in \mathcal{S}$ we recall that if $H$ is of type $S$, then $S$ denotes the socle of the subgroup $H$. For $n\leqs 12$ all of the non-geometric subgroups are known and can be found in the relevant tables in \cite[Section 8.2]{BHR}. Thus the proof for low dimensional cases follows a similar structure to the proof for the geometric subgroups. For $n\geqs 13$ the main tool here is a result of Guralnick et al. \cite{GPPS} which describes the subgroups $M$ of ${\rm GL}_n(q)$ such that $|M|$ is divisible by a primitive prime divisor of $q^i-1$ for $\frac{n}{2}<i\leqs n$ (see \cite[Examples 2.1-2.9]{GPPS}). This tells us the following 

\begin{prop}\label{p:gpps}
Suppose $n\geqs 13$, then $H_0$ has order divisible by a primitive prime divisor of $q^i-1$ where $\frac{n}{2}< i\leqs n$ only if either $S=A_m$ and $m=n+1,n+2$ or $(S,n,i)$ is found in Table \ref{tab:GPPStab}. 
\end{prop}

\begin{table}
\[
\begin{array}{lll} \hline
S & n & i \\ \hline
{\rm M}_{23} & 22& 22\\
{\rm M}_{24} & 23& 22\\
J_1 & 20 & 18\\
J_3 &18 &16,18\\
{\rm Co}_3 & 23 & 22\\
{\rm Co}_2 & 23 & 22\\
{\rm Co}_1 & 24 & 22\\
{\rm Ru} & 28& 28\\
{\rm Sz}(8) & 14 &12\\
{\rm G}_2(3) & 14 &12\\
{\rm PSp}_4(4) & 18 & 16\\
\Li_d(s), d\geqs 3 & \frac{s^d-1}{s-1}-1,\frac{s^d-1}{s-1} &  \frac{s^d-1}{s-1}-1\\[3pt]
\Un_d(s),d\geqs 3 & \frac{s^d+1}{s+1}-1,\frac{s^d+1}{s+1} &  \frac{s^d+1}{s+1}-1\\[3pt]
{\rm PSp}_{2d}(s) & \frac{1}{2}(s^n-1),\frac{1}{2}(s^n+1) &  \frac{1}{2}(s^n-1)\\[3pt]
{\rm PSp}_{2d}(3) &  \frac{1}{2}(3^n-1),\frac{1}{2}(3^n+1) & \frac{1}{2}(3^n-3)\\[3pt]
\Li_2(s) & s-1,s,s+1 & s-2  \\[2pt]
                       & s,s+1 & s \\[2pt]
                       & s-1,s,s+1 & s-1 \\[2pt]
                       &\frac{1}{2}(s-1),\frac{1}{2}(s+1) & \frac{1}{2}(s-1) \\[3pt]
                       &\frac{1}{2}(s-1),\frac{1}{2}(s+1) & \frac{1}{2}(s-3)\\[3pt]
\hline
\end{array}
\]
\caption{ The table for Lemma \ref{l:nongeom}}
\label{tab:GPPStab}
\end{table}

\begin{lem}
Theorem \ref{t:primedivs} holds if $H\in\mathcal{S}$, $n\geqs 13$ and $S=A_m$ with $m=n+1$ or $n+2$.
\end{lem}
\begin{proof}
Assume $S=A_m$ with $m=n+1$ or $n+2$. Note that all prime divisors of $|H_0|$ divide $(n+2)!$. However by assumption $n\geqs 13$ so $|G_0|$ is divisible by at least two primes larger than $n+2$ by Corollary \ref{c:r>n+2}. Thus $\pi(G_0)-\pi(H_0)\geqs 2$.
\end{proof}

\begin{lem}\label{l:nongeom}
Theorem \ref{t:primedivs} holds if $H\in\mathcal{S}$, $n\geqs 13$ and $S\neq A_m$ for $m=n+1$ or $n+2$ .
\end{lem}
\begin{proof}
Take $r_j$, $r_k$ and $r_l$ to be primitive prime divisors of $q^j-1$, $q^k-1$ and $q^l-1$ respectively, where $j:=2\lfloor \frac{n-1}{2} \rfloor$, $k:=2\lfloor \frac{n-3}{2} \rfloor$ and $l:=2\lfloor \frac{n-5}{2} \rfloor$. Note that $r_j,r_k$ and $r_l$ all exist and divide $|G_0|$, and $\frac{n}{2}<j,k,l<n$.
Proposition \ref{p:gpps} implies that if $(S,n)$ does not appear in Table \ref{tab:GPPStab}, $r_j$, $r_k$ and $r_l$ do not divide $|H_0|$ implying that $\pi(G_0)-\pi(H_0)\geqs 3$. Assume $(S,n,i)$ is found in Table \ref{tab:GPPStab} and $(S,n,i)\neq (\Li_2(s),s+1,s-2)$. By inspection of the table $n-2\leqs i\leqs n$, which implies that both $r_k$ and $r_l$ do not divide $|H_0|$, so $\pi(G_0)-\pi(H_0)\geqs 2$. 
Finally assume $(S,n,i)= (\Li_2(s),s+1,s-2)$. If $|H_0|$ is divisible by a primitive prime divisor of $q^t-1$ for $\frac{n}{2}<t\leqs n$ then $t=n-3$. Thus $r_j$ and $r_l$ do not divide $|H_0|$ so again $\pi(G_0)-\pi(H_0)\geqs 2$.
\end{proof}

\begin{lem}
Theorem \ref{t:primedivs} holds if $H\in\mathcal{S}$ and $n\leqs 12$.
\end{lem}
\begin{proof}
This can be shown by inspection of the appropriate Tables in \cite[Section 8.2]{BHR}. In most cases the proof is similar so we only provide the details for a handful of cases namely;
\begin{itemize}
\item[{\rm (a)}]$G_0={\rm P}\Omega^+_8(q)$ and $S={}^3D_4(q_0)$ with $q=q_0^3$.
\smallskip

\item[{\rm (b)}]$G_0=\Un_5(q)$ and $S=\Li_2(11)$.
\smallskip

\item[{\rm (c)}]$G_0={\rm PSp}_4(q)$ and $S={\rm Sz}(q)$ with $q\geqs 4$ even.
\smallskip

\item[{\rm (d)}]$G_0=\Li_2(q)$ and $S=A_5$.
\end{itemize} 

First we consider (a). Here $|H_0|=q_0^{12}(q_0^8+q_0^4+1)(q_0^6-1)(q_0^2-1).$
We immediately see that any primitive prime divisor of $q_0^{18}-1$ ($=q^6-1$) divides $|G_0|$ but not $|H_0|$. Take $s$ to be a primitive prime divisor of $q_0^4-1$. By Lemma \ref{l:Lemma A.1}, $s$ divides $q_0^{12}-1=q^3-1$, so $s$ divides $|G_0|$. Additionally, Lemma \ref{l:Lemma A.1} shows that $s\geqs 5$ and does not divide $q_0^{12}(q_0^6-1)(q_0^2-1)$. By Lemma \ref{l:Lemma A.4} we have that $(q_0^{12}-1)_s=(q_0^4-1)_s$, so $s$ does not divide $|H_0|$. Therefore $\pi(G_0)-\pi(H_0)\geqs 2$.

Next, let us turn to case (b). From \cite[Table 8.21]{BHR} we have $q=p\equiv 2,6,7,8,10\imod {11}$ and the prime divisors of $|H_0|$ are 2,3,5 and 11. Suppose $q\not\in\{2,7\}$ then by Lemma \ref{l:ppds4} there exist primitive prime divisors $r_4$ and $r_6$ of $q^4-1$ and $q^6-1$ respectively such that $r_4,r_6\geqs 13$, so $\pi(G_0)-\pi(H_0)\geqs 2$. The remaining cases $q=2$ and $7$ can be handled by a direct calculation showing that $\pi(G_0)-\pi(H_0)=0$ when $q=2$ and $\pi(G_0)-\pi(H_0)= 3$ when $q=7$.

Now consider case (c). By \cite[Table 8.14]{BHR}, $q=2^f$ with $f\geqs 3$ odd. Note that $|{\rm Sz}(q)|=q^2(q^2+1)(q-1)$ so any odd prime divisor of $q+1$ divides $|G_0|$ but not $|H_0|$. Therefore we may assume $q+1=r^l$ for some odd prime $r$ and some $l\geqs 1$. By Lemma \ref{l:btv} this occurs if and only if $f=3$, or $f=2^n$ and $r=2^f+1$ is a Fermat prime. However, $f\geqs 3$ is odd and so we reduce down to the case where $q=8$. Here it is straightforward to show that 3 is the only prime dividing $|G_0|$ that does not divide $|H_0|$. 

Finally we turn to case (d). From \cite[Table 8.2]{BHR} we have $q=p\equiv \pm 1\imod {10}$ or $q=p^2$ with $p\equiv \pm 3\imod {10}$. Note that the prime divisors of $|H_0|$ are 2,3 and 5. Suppose first that $p\geqs 7$. Then $p$ does not divide $|H_0|$ and the only other possible prime divisors of $|G_0|$ that do not divide $|H_0|$ are prime divisors $r$ of $q^2-1$ such that $r\geqs 7$. Therefore if $q^2-1=2^a.3^b.5^c$ for some $a,b,c\geqs 0$ we have $\pi(G_0)-\pi(H_0)=1$, otherwise $\pi(G_0)-\pi(H_0)\geqs 2$. Finally suppose $p\leqs 7$ then $q=9$, this case can be handled by direct computation showing that $\pi(G_0)-\pi(H_0)=0$.
\end{proof}

\section{Proof of Theorem \ref{t:mainthrm}}

In this final section we prove Theorem \ref{t:mainthrm}. First we apply Theorem \ref{t:primedivs} to reduce the problem to the cases listed in Tables \ref{tab:pi0}-\ref{tab:pi1}. Recall that $G\leqs{\rm Sym}(\Omega)$ is a finite almost simple primitive permutation group with socle $G_0\in \mathcal{G}$ where
\[
\mathcal{G}=\{\Li_n(q)\mid n\geqs 3\}\cup\{\Un_n(q), {\rm PSp}_n(q)\mid n\geqs 4\}\cup\{{\rm P}\Omega^{\epsilon}_n(q)\mid n\geqs 7\}\setminus \mathcal{F}.
\]
where $\mathcal{F}=\{\Li_3(2),\Li_4(2),{\rm PSp}_4(2)^{'},{\rm PSp}_4(3)\}$ and $q=p^f$ for $p$ prime and $f\geqs 1$.
Let $H$ be a point stabiliser and recall $H_0=H\cap G_0$. Additionally recall that $\pi(X)$ denotes the number of distinct prime divisors of $|X|$ and $P_a^b$ denotes the set of primitive prime divisors of $a^b-1$.

\begin{prop}\label{p:aereduc}
The following statements hold:
\begin{itemize}
\item[{\rm (i)}] If $G$ is almost elusive then $\pi(G_0)-\pi(H_0)\leqs 1$.
\item[{\rm (ii)}]If $(G_0,H,i)$ is found in Table \ref{tab:pi1} and $P_q^i=\{r\}$ then any element in $G_0$ of order $r$ is a derangement.
\end{itemize}
\end{prop}
\begin{proof}
Recall that $x\in G$ is a derangement if and only if $x^G\cap H=\emptyset$. Suppose $\pi(G_0)-\pi(H_0)\geqs 2$, then there exist distinct primes $r$ and $s$ that divide $|G_0|$ but not $|H_0|$. Thus every element of order $r$, or $s$, in $G_0$ is not an element of $H_0$. Since $G_0$ is a normal subgroup of $G$, any element of order $r$, or order $s$, in $G_0$ is a derangement. Therefore $G$ is not almost elusive. The second part of the proposition can easily be seen by using the same approach shown above.

\end{proof}

In view of Theorem \ref{t:primedivs}, we may assume for the remainder of the section that either $(G_0,H)$ is a case in Tables \ref{tab:pi0} or \ref{tab:pi2}, or $(G_0,H,i)$ is found in Table \ref{tab:pi1} and there exists a unique primitive prime divisor of $q^i-1$.

For certain low dimensional groups over $\mathbb{F}_q$ for small $q$, we can calculate the number of conjugacy classes of derangements of prime order directly using {\sc Magma} \cite{Mag}. Define $\mathcal{D}\cup\mathcal{O}\subset\mathcal{G}$ where 
$$\mathcal{D}=\{G_0\in\mathcal{G}\mid 3\leqs n \leqs 4 \text{ with } q\leqs 8 \text{ or } 5\leqs n\leqs 8 \text{ with } q=2\},$$ 
$$\mathcal{O}=\{\Omega_7(3),{\rm P}\Omega^{\epsilon}_8(3),\Omega^{\epsilon}_8(4),\Omega^{\epsilon}_{10}(2),\Omega^{\epsilon}_{12}(2)\}$$ 
The almost simple groups with socle in $\mathcal{D}\cup \mathcal{O}$ can be handled in {\sc Magma}. In particular, this handles most cases in Table \ref{tab:pi0} and all the relevant cases in Table \ref{tab:pi2}. 

\begin{prop}\label{p:smalldim}
Theorem \ref{t:mainthrm} holds for $G_0\in\mathcal{D}\cup\mathcal{O}$.
\end{prop}
\begin{proof}
This is an entirely straightforward {\sc Magma} \cite{Mag} calculation. For each $G_0\in\mathcal{D}\cup\mathcal{O}$ we use the command \verb|AutomorphismGroupSimpleGroup| to obtain $\Aut(G_0)$ as a permutation group. Then using \verb|LowIndexSubgroups| we obtain all almost simple groups with socle $G_0$. For each group we call \verb|MaximalSubgroups| and can identify the type of each maximal subgroup by its order or structure. Then for each group $G$ and each maximal subgroup $H$ of $G$ we look over all conjugacy classes of prime order elements in $G$ and see how many intersect with an $H$-class. 
\end{proof}

\begin{cor}
Theorem \ref{t:mainthrm} holds for cases ${\rm I-X}$, ${\rm XII}$ with $(n,q)=(6,2)$, ${\rm XIII}$, ${\rm XIV}$ and ${\rm XIX}$ in Table \ref{tab:pi0} and for all cases in Table \ref{tab:pi2}.
\end{cor}

It now remains to deal with cases XI, XII with $n\equiv 0\imod 4$, XV-XVIII and XX in Table \ref{tab:pi0} and all remaining cases in Table \ref{tab:pi1} (that is all cases except for L4, L6, U2 and U7 since in these cases $G_0\not\in \mathcal{G}$). 
Aschbacher's theorem \cite{asch} provides the framework for the proof. The subspace subgroups comprising the $\mathcal{C}_1$ collection require special attention. These are relatively large subgroups so finding prime order derangements can be more challenging.

Before we begin the proofs we provide some preliminary lemmas that will be useful for both non-subspace and subspace subgroups. We remind the reader that the notation for elements of classical groups was set up in Section \ref{s:conjclass}. The following lemmas are particular cases of \cite[Lemma 4.2.4]{BG_book}.

\begin{lem}\label{l:tisSO}
Suppose $G_0\in\{{\rm PSp}_n(q),{\rm P}\Omega_n^{\epsilon}(q)\}$ where $n\equiv 0\imod 4$. Suppose $r$ is a primitive prime divisor of $q^{n/2}-1$, and let $x=\hat{x}Z\in G_0$ be an element of order $r$ such that $\hat{x}:=[\Lambda,I_{n/2}]$. Then $x$ does not fix a totally singular $n/2$-space.
\end{lem}

\begin{lem}\label{l:tisU}
Suppose $G_0=\Un_n(q)$ and let $n\geqs 6$ be even such that $n\not\equiv 0\imod 8$. Suppose $r$ is a primitive prime divisor of $q^{i}-1$, where $i=n/2$ if $n\equiv 4\imod 8$ and $i=n$ if $n\equiv \pm2\imod 8$. Let $x=\hat{x}Z\in G_0$ be an element of order $r$ such that 
\begin{equation}\label{e:tis}
\hat{x}:=
\begin{cases}
[\Lambda_1^3,\Lambda_2] & n\equiv 4\imod 8\\
[\Lambda_1,\Lambda_2] & n\equiv \pm2\imod 8
\end{cases},
\end{equation}
where $\Lambda_1\neq \Lambda_2$. Then $x$ does not fix a totally singular $n/2$-space.
\end{lem}

\subsection{Non-subspace subgroups} \label{ss:nonsub}

In this section we prove Theorem \ref{t:mainthrm} for the maximal subgroups contained in one of the collections $\mathcal{C}_2,\dots,\mathcal{C}_8$, and those contained in the $\mathcal{N}$ and $\mathcal{S}$ collections (see Section \ref{s:subgrp struc} for more details on these subgroup collections). 

\begin{lem}\label{l:symppd4}
Let $G_0={\rm PSp}_n(q)$ such that $n=4,6$ and suppose $P_q^4=\{r\}$. Then either $G_0$ contains at least two $G$-classes of elements of order $r$ or one of the following holds;
\begin{itemize}
\item[{\rm (i)}] $(n,q)=(4,3), (4,7), (6,2), (6,3), (6,7)$ and $r=5$; or 
\item[{\rm (ii)}] $(n,q)=(4,4)$, $G=\Aut(G_0)$ and $r=17$.
\end{itemize}
\end{lem}
\begin{proof}
Here any element $x=\hat{x}Z\in G_0$ of order $r$ is going to have the form $\hat{x}=[\Lambda]$ if $n=4$ and $\hat{x}=[\Lambda,I_2]$ if $n=6$ (see Section \ref{s:conjclass} and \cite[Section 3.4]{BG_book}). Since $r$ is the unique primitive prime divisor of $q^4-1$, it is also the unique primitive prime divisor of $p^{4f}-1$, so $r=4fd+1$ for some $d\geqs 1$. Therefore $G_0$ contains $(r-1)/4= fd$ distinct ${\rm PGSp}_n(q)$-classes of elements of order $r$ (see \cite[Proposition 3.4.3]{BG_book}). Since $|\Aut(G_0){:}{\rm PGSp}_n(q)|=kf$ where $k=2$ if $(n,p)=(4,2)$ and 1 otherwise, it follows that there are at least $(r-1)/4kf\geqs d/2$ distinct $G$-classes of elements of order $r$ in $G_0$.

By Lemma \ref{l:ppds4} either $d\geqs 4$ or $(d,q)=(3,5),(3,239),(2,4),(1,2),(1,3)$ or $(1,7)$. Suppose first that $d\geqs 3$ then by the argument above there are at least 2 distinct $G$-classes in $G_0$ of elements of order $r$. For the remaining cases, that is $q\in\{2,3,4,7\}$, it is a straightforward {\sc Magma} calculation.
\end{proof}

Using this result we can immediately handle a large number of the cases with symplectic socle in Table \ref{tab:pi1}.

\begin{prop}\label{p:symp}
Theorem \ref{t:mainthrm} holds for cases {\rm S4}, {\rm S5}, {\rm S7}, {\rm S10} and {\rm S11} in Table \ref{tab:pi1}.
\end{prop}
\begin{proof}
In all of these cases $G_0={\rm PSp}_n(q)$ such that $n=4$ or $6$. Additionally we are assuming that there exists a unique primitive prime divisor, $r$, of $q^4-1$ which is the unique prime dividing $|G_0|$ and not $|H_0|$. Note if $n=6$ we are in case S10 and $q\geqs 4$ is even. Additionally, by Proposition \ref{p:smalldim}, if $n=4$ then we can assume $q\geqs 9$. Thus by Lemma \ref{l:symppd4}, $G$ contains at least two conjugacy classes of derangements of order $r$, so $G$ not almost elusive.
\end{proof}

 We organise the rest of this subsection by the relevant Aschbacher collections. Before we begin we state some useful notation; 

\begin{nota}
Assume $(G_0,H,i)$ is a case in Table \ref{tab:pi1}. Then we use $r_i$ to denote the unique primitive prime divisor of $q^i-1$.  
\end{nota}
Additionally we recall that if $(G_0,H,i)$ is found in Table \ref{tab:pi1} then by Proposition \ref{p:aereduc} any element of order $r_i$ in $G_0$ is a derangement.

\subsubsection{$\mathcal{C}_2$ subgroups}

\begin{prop}
Theorem \ref{t:mainthrm} holds for case {\rm U4} in Table \ref{tab:pi1}.
\end{prop}
\begin{proof}
Here $(G_0,i)=(\Un_n(q),2n-2)$ with $n=4$ or $6$, and $H$ is the stabiliser in $G$ of a decomposition $V=V_1\oplus V_2$, where the $V_j$ are both maximal totally singular spaces of dimension $n/2$. Recall by Proposition \ref{p:smalldim} we may assume $q\geqs 3$.

Suppose first that $n=6$. Take $s$ to be a primitive prime divisor of $q^{6}-1$ and let $x=\hat{x}Z\in G_0$ be an element of order $s$ where $\hat{x}$ is defined as in \eqref{e:tis}. Then $x$ does not fix $V_1$ or $V_2$ by Lemma \ref{l:tisU} and does not interchange $V_1$ and $V_2$ since $|x|=s>2$. Thus $x$ is a derangement of order $s$, so $G$ is not almost elusive. 

Finally suppose $n=4$. In this case, by Proposition \ref{p:smalldim}, we may assume $q\geqs 9$, thus either $q=19$ or $r_6\geqs 13$ by Lemma \ref{l:ppd2a3}. Assume $q\neq 19$, then $\Un_4(q)$ contains $(r_6-1)/3\geqs 4$ distinct ${\rm PGU}_4(q)$-classes of elements of order $r_6$ in $G_0$. Since $|\Aut(G_0){:}{\rm PGU}_4(q)|=2$ there are at least $(r_6-1)/6\geqs 2$ distinct $G$-classes of elements of order $r_6$ in $\Un_4(q)$, so $G$ is not almost elusive. Finally assume $q=19$ and take $x=\hat{x}Z\in G_0$ to be an element of order 5 (the unique primitive prime divisor of $q^2-1$) such that $\hat{x}=[\Lambda,\Lambda^{-1},I_2]$. Since the eigenvalues of $\hat{x}$ (on $V\otimes \mathbb{F}_{q^2}$) have odd multiplicity, $x$ is a derangement.
\end{proof}

\begin{prop}
Theorem \ref{t:mainthrm} holds for case {\rm O8} in Table \ref{tab:pi1}.
\end{prop}
\begin{proof}
In this case $(G_0,i)=({\rm P}\Omega^+_8(q),6)$, and $H$ is the stabiliser in $G$ of a decomposition $V=V_1\oplus V_2$, where the $V_j$ are both maximal totally singular spaces of dimension $4$. 

Take $s$ to be a primitive prime divisor of $q^{4}-1$. Let $x=\hat{x}Z\in G_0$ be an element of order $s$ where $\hat{x}=[\Lambda,I_{n/2}]$. Then $x$ does not fix $V_1$ or $V_2$ by Lemma \ref{l:tisSO}, and does not interchange $V_1$ and $V_2$ since $|x|=s>2$. Thus $x$ is a derangement. It follows that $G$ contains derangements of order $r_6$ and $s$, so is not almost elusive. 
\end{proof}

\subsubsection{$\mathcal{C}_3$ subgroups}

For the following lemma we recall that $V$ denotes the natural $G_0$-module. 

\begin{lem}\label{l:c3}
Let $H\in \mathcal{C}_3$ be a subgroup arising from a field extension of prime degree $k$. Let $x\in H_0$ be an element of order $p$. Then either;
\begin{itemize}
\item[{\rm (i)}] $x$ has Jordan form $[J_p^{ka_p},\dots ,J_1^{ka_1}]$ on $V$; or 
\item[{\rm (ii)}] $k=p$ and $x$ has Jordan form $[J_k^{n/k}]$ on $V$.
\end{itemize}
\end{lem}
\begin{proof}
This follows by applying Lemmas 5.3.2 and 5.3.11 in \cite{BG_book}.
\end{proof}

\begin{prop}
Theorem \ref{t:mainthrm} holds for cases {\rm L5} and {\rm S6} in Table \ref{tab:pi1}.
\end{prop}
\begin{proof}
In both cases we know that $G_0$ contains derangements of order $r_i$. By Lemma \ref{l:c3} any element in $G_0$ of order $p$ with Jordan form $[J_2,J_1^{n-2}]$ on $V$ is also a derangement. Thus $G_0$ contains derangements of order $p$ and $r_i$, so $G$ is not almost elusive.
\end{proof}

\begin{prop}
Theorem \ref{t:mainthrm} holds for case {\rm O9} in Table \ref{tab:pi1}.
\end{prop}
\begin{proof}
Here $(G_0, i)=( {\rm P}\Omega_8^+(q), 3)$ and $H$ of type ${\rm GU}_{4}(q)$. 
Assume first that $p\neq 2$. Then any element in $G_0$ of order $p$ with Jordan form $[J_3,J_1^5]$ on $V$ is a derangement by Lemma \ref{l:c3}. Thus $G_0$ contains semisimple and unipotent derangements, so $G$ is not almost elusive. 

Finally assume $p=2$. Note that $H_0=C_{q+1}.{\rm PGU}_4(q).\langle \psi \rangle$, where $\psi$ is an involutary graph automorphism of $\Un_4(q)$ arising from an involutary field automorphism of $\mathbb{F}_{q^2}$ (see \cite[Lemma 5.3.6]{BG_book}). If $x\in H_0$ is an involution with Jordan form $[J_2^2,J_1^4]$ on $V$, then $x$ has Jordan form $[J_2,J_1^2]$ on the natural $\Un_4(q)$-module. In $G_0$ there are precisely two $G_0$-classes of involutions with Jordan form $[J_2^2,J_1^4]$ on $V$ (these are represented by the elements $a_2$ and $c_2$; see \cite[Section 3.5.4]{BG_book} for more details). However there is a unique class of involutions in $H_0$ with Jordan form $[J_2,J_1^2]$ on the natural $\Un_4(q)$-module (see \cite[Proposition 3.3.7]{BG_book}), so we conclude that $G_0$ must contain derangements of order $p=2$. Thus $G$ is not almost elusive. 
\end{proof}

\begin{prop}
Theorem \ref{t:mainthrm} holds for case ${\rm XI}$ in Table \ref{tab:pi0}.
\end{prop}
\begin{proof}
Here $G_0={\rm PSp}_4(q)$ and $H$ is of type ${\rm Sp}_2(q^2)$, so $H_0={\rm PSp}_2(q^2).2$ (see \cite[Proposition 4.3.10]{KL}). Note we may assume $q\geqs 8$ by Proposition \ref{p:smalldim}. From Lemma \ref{l:c3} any element in $G_0$ of order $p$ with Jordan form $[J_2,J_1^2]$ is a derangement.
Take $r$ to be a primitive prime divisor of $q^i-1$ where $i=1$ if $q$ is a Mersenne prime and $i=2$ otherwise (note $r$ always exists by Lemma \ref{l:btv} and Zsigmondy's theorem). Let $x=\hat{x}Z\in G_0$ be an element of order $r$ such that 
\begin{equation*}
\hat{x}:=
\begin{cases}
[\Lambda,I_2] & i=2\\
[\Lambda,\Lambda^{-1},I_2] &i=1.
\end{cases}
\end{equation*} 
Then \cite[Lemma 5.3.2]{BG_book} implies that any element of order $r$ in $H_0$ must have a trivial one-eigenspace, implying that $x$ is a derangement. Therefore the result holds.
\end{proof}

\subsubsection{$\mathcal{C}_5$ subgroups}

\begin{prop}
Theorem \ref{t:mainthrm} holds for case ${\rm U5}$ in Table \ref{tab:pi1}.
\end{prop}
\begin{proof}
Here $(G_0, i)=(\Un_n(q), 2n-2)$ with $n=4$ or $6$, and $H$ is of type ${\rm Sp}_n(q)$. Note that by Proposition \ref{p:smalldim} we can assume $q\geqs 9$ when $n=4$ and $q\geqs 3$ when $n=6$.

Assume $p\geqs 3$ and take $x\in G_0$ to be an element of order $p$ with Jordan form $[J_3,J_1^{n-3}]$, then $x\not\in H_0$ (recall all odd sized Jordan blocks must have even multiplicity, see Lemma \ref{l:osjordan}). Thus for $p\geqs 3$ we are done. 

Now assume $p=2$. Suppose $n=6$ and take $s$ to be a primitive prime divisor of $q^6-1$. Take $x=\hat{x}Z\in G_0$ to be an element of order $s$, such that $\hat{x}=[\Lambda,I_3]$. Then $x\not\in H$ since $\Lambda\neq \Lambda^{-1}=\{\lambda^{-1}\mid \lambda\in\Lambda\}$ (see \cite[Proposition 3.4.3]{BG_book}), so $x$ is a derangement. Similarly suppose $n=4$ and take $s$ to be a primitive prime divisor of $q^2-1$. Let $x=\hat{x}Z\in{\rm PGU}_4(q)$ be an element of order $s$ such that $\hat{x}=[\lambda I_1,I_3]$. Then by the same reasoning as for the $n=6$ case, $x$ is a derangement, and $x\in G_0$ since $(4,s)=1$ (see \cite[Proposition 3.2.2]{BG_book}). Thus $G$ is not almost elusive.
\end{proof}

\begin{prop}
Theorem \ref{t:mainthrm} holds for case ${\rm U6}$ in Table \ref{tab:pi1}.
\end{prop}
\begin{proof}
In this case $(G_0,i)=(\Un_4(q),6)$ and $H$ is of type ${\rm O}^{-}_4(q)$ with $q$ odd. Take $x\in G_0$ to be an element of order $p$ with Jordan form $[J_2,J_1^{n-2}]$. Then $x$ is a derangement (recall all even sized blocks in the Jordan form of a unipotent element in $H_0$ must have even multiplicity, see Lemma \ref{l:osjordan}). Thus $G$ is not almost elusive, since $G_0$ contains derangements of order $p$ and $r_6$ (the unique primitive prime divisor of $q^6-1$).
\end{proof}

\begin{prop}
Theorem \ref{t:mainthrm} holds for case ${\rm O10}$ in Table \ref{tab:pi1}.
\end{prop}
\begin{proof}
Here $(G_0,i)=({\rm P}\Omega_8^{+}(q),6)$ and $H$ is of type ${\rm O}^{-}_8(q_0)$, where $q=q_0^2$. 
Take $s$ to be a primitive prime divisor of $q_0^8-1$ and note that $s$ is also a primitive prime divisor of $q^4-1$. Let $x=\hat{x}Z\in G_0$ be an element of order $s$ such that $\hat{x}=[\Lambda,I_4]$. By \cite[Proposition 3.5.4]{BG_book} any element $y=\hat{y}Z'\in H_0$ of order $s$ must have the form $\hat{y}=[\Lambda',I_1]$, where $Z'=Z(O^-_8(q_0))$, $\Lambda'=\{\lambda,\lambda^{q_0},\dots,\lambda^{q_0^7}\}$ and $\lambda$ is some non-trivial $s^{th}$ root of unity in $\mathbb{F}_{q^4}$. Thus $x\not\in H_0$ since it has a 4-dimensional 1-eigenspace, so $x$ is a derangement. Thus $G_0$ contains derangements of distinct prime order, implying that $G$ is not almost elusive.  
\end{proof}

\subsubsection{$\mathcal{C}_8$ subgroups}

\begin{prop}
Theorem \ref{t:mainthrm} holds for case ${\rm L7}$ in Table \ref{tab:pi1}.
\end{prop}
\begin{proof}
Here $(G_0,i)=({\rm L}_n(q),n-1)$ with $H$ of type ${\rm Sp}_n(q)$ and $n=4$ or $6$. Assume first that $p\geqs 3$ and let $x\in G_0$ be an element of order $p$ with Jordan form $[J_3,J_1^{n-3}]$. Then $x$ is a derangement since all odd sized blocks in symplectic groups must have even multiplicity (see Lemma \ref{l:osjordan}). Thus for the remainder of the proof we may assume that $p=2$. By Proposition \ref{p:smalldim} we may assume that $q\geqs 16$ when $n=4$, and $q\geqs 4$ when $n=6$. Thus $r_{n-1}\geqs 4(n-1)f+1$ by Lemmas \ref{l:ppd2a3} and \ref{l:lic1prime}. In the usual manner $G_0$ contains $(r_{n-1}-1)/(n-1)=4f$ distinct ${\rm PGL}_n(q)$-classes of elements of order $r_{n-1}$. Since $|\Aut(G_0){:}{\rm PGL}_n(q)|=2f$, there are at least $2$ distinct $G$-classes of elements of order $r_{n-1}$ in $G_0$.  
\end{proof}

\begin{prop}
Theorem \ref{t:mainthrm} holds for case ${\rm L8}$ in Table \ref{tab:pi1}.
\end{prop}
\begin{proof}
Here $(G_0,i)=({\rm L}_n(q),3)$ and $H$ is of type ${\rm O}^{\epsilon}_n(q)$ such that $(\epsilon,n)=(\circ, 3),(-,4)$ and $q$ is odd. Take $x\in G_0$ be a unipotent element with Jordan form $[J_2,J_1^{n-2}]$. Then $x$ is a derangement (since in orthogonal groups even sized Jordan blocks must have even multiplicity). Thus $G_0$ contains unipotent and semisimple derangements, so $G$ is not almost elusive.
\end{proof}

\begin{prop}\label{p:S7}
Theorem \ref{t:mainthrm} holds for cases ${\rm S8}$ and ${\rm S9}$ in Table \ref{tab:pi1} and case ${\rm XII}$ in Table \ref{tab:pi0}.
\end{prop}
\begin{proof}
Here $G_0={\rm PSp}_n(q)$ with $q$ even and $H$ is of type ${\rm O}^{\epsilon}_n(q)$ (recall $n\geqs 4$ since $G_0\in\mathcal{G}$). The cases $\epsilon=+$ and $\epsilon=-$ are similar, so we only provide details for the $\epsilon=+$ case.

Assume $\epsilon=+$. Take $s$ to be an odd prime divisor of $q^{n/2}+1$ and let $j$ be such that $s$ is a primitive prime divisor of $q^j-1$. Then $j$ divides $n$ and does not divide $n/2$, implying $j$ is even and $n/j$ is odd. 
Let $x=\hat{x}Z\in G_0$ be an element of order $s$ such that $\hat{x}=[\Lambda^{n/j}]\in{\rm Sp}_n(q)$, then by \cite[Remark 3.5.5]{BG_book} $x\not\in {\rm O}^{+}_n(q)$, so is a derangement. Thus we may assume that $q^{n/2}+1=s^l$ for some $l\geqs 1$. By Lemma \ref{l:btv} this occurs if and only if one of the following holds; 
\begin{itemize}
\item[{\rm (i)}] $(n,q)=(6,2)$
\item[{\rm (ii)}] $fn=2^m$ with $m\geqs 2$, and $s$ is a Fermat prime. 
\end{itemize}

The case (i) was handled in Proposition \ref{p:smalldim} thus we may assume that $(n,q)$ is as in case (ii). Here $s=2^{fn/2}+1$ and $G_0$ contains $(s-1)/n$ distinct ${\rm PGSp}_n(q)$-classes of elements of order $s$ (see \cite[Section 3.4.1]{BG_book}). Since $|\Aut(G_0){:}{\rm PGSp}_n(q)|=af$ where $a=2$ when $n=4$ and $a=1$ otherwise (see \cite[Section 2.4]{BG_book}), there are at least $c:=(s-1)/afn$ distinct $G$-classes of elements of order $s$ in $G_0$. It is straightforward to see that $c=2^{fn/2}/afn\geqs 2$ for $(n,q)\neq(4,2),(4,4)$. Thus $G$ is not almost elusive when $(n,q)\neq(4,2),(4,4)$. The remaining cases $(n,q)=(4,2),(4,4)$ have been handled already in Proposition \ref{p:smalldim}.
\end{proof}

We have now handled the cases in which $H$ is contained in one of the Aschbacher collections $\mathcal{C}_2,\dots,\mathcal{C}_8$. Thus to complete the proof of Theorem \ref{t:mainthrm} for $H\not\in \mathcal{C}_1$ it remains for us to handle the remaining subgroups in $\mathcal{N}$ (the novelty subgroups) and $\mathcal{S}$ (the non-geometric subgroups). That is cases ${\rm XVII}$ and ${\rm XVIII}$ in Table \ref{tab:pi0} and cases ${\rm O11}$ and ${\rm O20}$ in Table \ref{tab:pi1} (recall that cases S10 and S11 have been handled already in Proposition \ref{p:symp}). 

\subsubsection{Novelty subgroups}

\begin{prop}
Theorem \ref{t:mainthrm} holds for case ${\rm O11}$ in Table \ref{tab:pi1}.
\end{prop}
\begin{proof}
In this case $(G_0,i)=({\rm P}\Omega^{+}_8(q),4)$ and $H_0=H\cap G_0=G_2(q)$. By \cite[Proposition 3.1.1 (vi)]{Klei} every element of $H_0$ fixes a nonsingular 1-space (a reducible subgroup of type ${\rm O}_7(q)$). Thus any element in $G_0$ that does not fix a nonsingular 1-space is a derangement. Therefore $G$ is not almost elusive by Propositions \ref{p:spn-1} and \ref{p:nd1}.
\end{proof}

\subsubsection{Non-geometric subgroups}

Here we handle the remaining subgroups $H\in \mathcal{S}$. We recall that here \textit{type of} $H$ refers to the socle of $H$. We will use $S$ to denote type of $H$, that is $S={\rm Soc}(H)$. 

\begin{prop}
Theorem \ref{t:mainthrm} holds for cases ${\rm XVII}$ and ${\rm XVIII}$ in Table \ref{tab:pi0}.
\end{prop}
\begin{proof}
Here $G_0={\rm P}\Omega_8^+(q)$ and $S=H_0=\Omega_7(q)$ if $q$ is odd or $S=H_0={\rm Sp}_6(q)$ if $q$ is even. By \cite[Proposition 2.2.4]{Klei} there exists a triality graph automorphism $\tau$ of $G_0$ such that $H_0^{\tau}$ is a $\mathcal{C}_1$ subgroup of type ${\rm O}_1(q)\perp{\rm O}_7(q)$ when $q$ is odd and ${\rm Sp}_6(q)$ (a stabiliser in $G_0$ of a nonsingluar 1-space) if $q$ is even. Then, again $G$ is shown to be not almost elusive in Propositions \ref{p:spn-1} and \ref{p:nd1}. 
\end{proof}

\begin{prop}
Theorem \ref{t:mainthrm} holds for case ${\rm O20}$ in Table \ref{tab:pi1}.
\end{prop}
\begin{proof}
Here $(G_0,i)=(\Omega_7(q),4)$ and $S={\rm G}_2(q)$. By Proposition \ref{p:smalldim} we may assume that $q\geqs 5$. Since $r_4$ is the unique primitive prime divisor of $q^4-1$, by Lemma \ref{l:ppds4} we may assume that either $q=7$ and $r_4=5$ or $r_4\geqs 12f+1$. 
Assume first that $r_4\geqs 12f+1$. Then $G_0$ contains $(r_4-1)/4\geqs3f$ distinct ${\rm PGO}_7(q)$-classes of derangements of order $r_4$. Additionally since $|\Aut(G_0){:}{\rm PGO}_7(q)|=f$ we have that $G_0$ contains at least $(r_4-1)/4f\geqs 3$ distinct $G$-classes of derangements of order $r_4$ and so the result follows. 
Finally assume $q=7$. It is straightforward to check using {\sc Magma} that there are three conjugacy classes of semisimple involutions in $G_0$, and that there is a unique class of involutions in ${\rm G}_2(q)$. Therefore we conclude that $G_0$ contains derangements of order 2 and $r_4$, so the result follows.  
\end{proof}

In view of all the propositions proved in this section and Proposition \ref{p:smalldim} we have shown the following; 
\begin{prop}\label{p:allnonsub}
Let $G\leqs{\rm Sym}(\Omega)$ be a finite almost simple primitive permutation group with classical socle $G_0\in\mathcal{G}$ and point stabiliser $H\not\in \mathcal{C}_1$. Then $G$ is almost elusive if and only if $(G,H)$ is a case recorded in Table \ref{tab:maintab} with $H\not\in \mathcal{C}_1$.
\end{prop}

\subsection{Subspace subgroups}

Here we complete the proof of Theorem \ref{t:mainthrm} by handling the cases when $H$ is a subspace subgroup, that is $H$ contained in the $\mathcal{C}_1$ Aschbacher subgroup collection. Once again we recall that when $(G_0,H,i)$ is a case in Table \ref{tab:pi1} we are assuming there exists a unique primitive prime divisor $r_i$ of $q^i-1$, and any element in $G_0$ of order $r_i$ is a derangement.  

\subsubsection{Symplectic groups}

The remaining cases with $G_0={\rm PSp}_n(q)$ are the following cases in Table \ref{tab:pi1}; 
\begin{itemize}
\item[{\rm (a)}] Case S1: $H$ is of type ${\rm P}_1$ and $i=n$ with $n\equiv 0\imod4$;
\item[{\rm (b)}] Case S2: $H$ is of type ${\rm P}_2$ and $i=n$ with $n=4$;
\item[{\rm (c)}] Case S3: $H$ is of type ${\rm Sp}_2(q)\perp{\rm Sp}_{n-2}(q)$ and $i=n$ with $n\equiv 0\imod4$.
\end{itemize}

\begin{prop}
Theorem \ref{t:mainthrm} holds for cases ${\rm S1}$, ${\rm S2}$ and ${\rm S3}$ in Table \ref{tab:pi1} with $n=4$.
\end{prop}
\begin{proof}
By Proposition \ref{p:smalldim} we may assume that $q\geqs 9$, so either $q=239$ and $r_4=13$ or $r_4\geqs 16f+1$ by Lemma \ref{l:ppds4}. 
First suppose that $r_4\geqs 16f+1$. Then there are $(r_4-1)/4\geqs 4f$ distinct ${\rm PGSp}_4(q)$-classes of derangements of order $r_4$ in $G_0$ (see \cite[Proposition 3.4.3]{BG_book}). It follows that there are at least $(r_4-1)/8f\geqs 2$ distinct $G$-classes of derangements of order $r_4$ in $G_0$ since $| \Aut(G_0){:}{\rm PGSp}_4(q)|\leqs 2f$. 

Finally suppose that $q=239$. Then by a similar argument to before (noting that in this case $| \Aut(G_0){:}{\rm PGSp}_4(q)|=f=1$) we conclude that there are at least $(r_4-1)/4=3$ distinct $G$-classes of derangements of order $r_4=13$ in $G_0$, so $G$ is not almost elusive. 
\end{proof}

\begin{prop}
Theorem \ref{t:mainthrm} holds for cases ${\rm S1}$ and ${\rm S3}$ in Table \ref{tab:pi1} with $n\geqs 8$.
\end{prop}
\begin{proof}

Assume first $(n,q)\neq (12,2)$ and take $s$ to be a primitive prime divisor of $q^{n/2}-1$. Let $x=\hat{x}Z\in G_0$ be an element of order $s$ such that $\hat{x}=[\Lambda^2]$. Since $\dim C_V(\hat{x})=0$ and $|\Lambda|=n/2$, $x$ does not fix a 1-space or a 2-space and thus is a derangement. Therefore $G$ contains derangements of order $s$ and $r_n$ (the unique primitive prime divisor of $q^n-1$).

For the final case $(n,q)=(12,2)$ it is easy to see that elements in $G_0$ of order 13 of the form $[\Lambda]Z$ are derangements. Similarly elements $[\Lambda^3]Z\in G_0$ of order 5 are also derangements.
\end{proof}

\subsubsection{Linear groups}

We begin by recalling that throughout this subsection we will be taking $n\geqs 3$ (the cases for $n=2$ were handled in a previous paper  \cite{BHall}). Before we handle the remaining linear group cases we first provide a result on the number of conjugacy classes of elements of certain orders in $G_0$. In the following lemma we let $K_{G_0}(G,r)$ denote the number of $G$-classes of elements of order $r$ in $G_0$ and $\phi$ denotes a field automorphism of $\Li_n(q)$ of order $f$.

\begin{lem}\label{l:classeslinear}
Let $G_0=\Li_n(q)$ where $q=p^f$. Suppose $r$ is a primitive prime divisor of $p^{fn}-1$ such that $r=knf+1$ for some positive integer $k$. Then $K_{G_0}(G,r)\geqs k/2$ for any group $G\leqs \Aut(G_0)$. In particular if $G\leqs \langle {\rm PGL}_n(q),\phi\rangle$ then $K_{G_0}(G,r)\geqs k$.
\end{lem}
\begin{proof}
By Lemma \ref{l:prime fac of f} we know that $r$ is also a primitive prime divisor of $q^n-1$, so any element in $G_0$ of order $r$ must have the form $x=[\Lambda]Z$ (see Section \ref{s:conjclass}). Thus by \cite[Proposition 3.2.1]{BG_book}, $G_0$ contains $(r-1)/n=kf$ distinct ${\rm PGL}_n(q)$-classes of elements of order $r$ since $n\geqs3$. Note that $|\Aut(G_0){:}{\rm PGL}_n(q)|=2f$ and $|\langle {\rm PGL}_n(q),\phi\rangle{:}{\rm PGL}_n(q)|=f$, so the result follows.
\end{proof}

The remaining cases left to handle with $G_0=\Li_n(q)$ are the following cases found in Table \ref{tab:pi1}; 
\begin{itemize}
\item[{\rm (a)}] Case L1: $H$ is of type ${\rm P}_1$ and $i=n$;
\item[{\rm (b)}] Case L2: $H$ is of type ${\rm GL}_1(q)\oplus {\rm GL}_{n-1}(q)$ and $i=n$;
\item[{\rm (c)}] Case L3: $H$ is of type ${\rm P}_{1,n-1}$ and $i=n=3$ and $q=p$ a Mersenne prime.
\end{itemize}

\begin{prop}
Theorem \ref{t:mainthrm} holds for case ${\rm L3}$ in Table \ref{tab:pi1}.
\end{prop}
\begin{proof}
Here by assumption $q=p$ is a Mersenne prime, so $r_3\geqs 13=4nf+1$ by Lemma \ref{l:ppd2a3}. Therefore $G_0$ contains at least 2 distinct $G$-classes of derangements of order $r_3$ by Lemma \ref{l:classeslinear}, so $G$ is not almost elusive. 
\end{proof} 

\begin{prop}
If $n$ is composite, then Theorem \ref{t:mainthrm} holds for cases ${\rm L1}$ and ${\rm L2}$ in Table \ref{tab:pi1}.
\end{prop}
\begin{proof}
Write $n=jh$ such that $j,h\neq 1$. First assume $n=4$ and $q=p$ is a Mersenne prime. Note that we may assume $q>8$ by Proposition \ref{p:smalldim}. Additionally we note $r_4$ is also a primitive prime divisor of $p^{4f}-1$, and $r_4\geqs 4nf+1$ by Lemma \ref{l:ppds4}. Thus by Lemma \ref{l:classeslinear}, $G_0$ contains at least 2 distinct $G$-classes of derangements of order $r_4$, so $G$ is not almost elusive. In the remaining cases without loss of generality there always exists a primitive prime divisor, $s$, of $q^h-1$. Take $x=\hat{x}Z\in G_0$ to be an element of order $s$ such that $\hat{x}=[\Lambda^s]$. Then $x$ does not fix a 1-dimensional subspace of $V$, so $x$ is a derangement. Therefore $G$ is not almost elusive since $G_0$ contains derangements of order $s$ and $r_n$. 
\end{proof}

\begin{prop}\label{p:lic1prime}
Assume $n$ is prime. Then Theorem \ref{t:mainthrm} holds for cases ${\rm L1}$ and ${\rm L2}$ in Table \ref{tab:pi1}.
\end{prop}
\begin{proof}
Since $r_n$ is the unique primitive prime divisor of $q^n-1$, it must also be the unique primitive prime divisor of $p^{fn}-1$ (note by Proposition \ref{p:smalldim} we may assume $P_p^{fn}\neq\emptyset$). Thus it follows that $r_n=kfn+1$ for some $k\geqs 1$. By Lemma \ref{l:prime fac of f}, $f=n^j$ for some $j\geqs 0$ and since $r_n$ and $n$ are both odd primes it follows that $k\geqs 2$ is even. 

Assume first that case L1 holds, that is $H$ is of type $P_1$. To ensure maximality of $H$ we require $G\leqs \langle {\rm PGL}_n(q),\phi\rangle$ where $\phi$ is a field automorphism of $\Li_n(q)$ of order $f$ (this is since the inverse-transpose graph automorphism interchanges the stabilisers of $m$-spaces and $(n-m)$-spaces). Thus by Lemma \ref{l:classeslinear} we conclude that $G$ is not almost elusive. 

For the remainder of the proof we may assume that case L2 holds, that is $H$ is of type ${\rm GL}_1(q)\oplus{\rm GL}_{n-1}(q)$. 
First suppose $p\geqs 3$ and let $x=\hat{x}Z\in H_0=H\cap G_0$ be an element of order $p$. Then $\hat{x}\in {\rm GL}_1(q) \oplus {\rm GL}_{n-1}(q)$ is ${\rm GL}_n(q)$-conjugate to $[J_p^{a_p},\dots,J_2^{a_2},J_1^{a_1+1}]$ with $a_t\geqs 0$ for all $t$ and $\sum_{t=1}^p ta_t=n-1$. Therefore any element in $G_0$ of order $p$ with Jordan form $[J_3,J_2^{(n-3)/2}]$ is a derangement. Thus $G$ is not almost elusive since $G_0$ contains derangements of order $r_n$ and $p$. 
Finally suppose $p=2$ and recall that if $n=3$ then by Proposition \ref{p:smalldim} we may assume $q\geqs 9$. Thus $k\geqs 4$ by Lemma \ref{l:lic1prime}, so by Lemma \ref{l:classeslinear} $G$ is not almost elusive. 
\end{proof}

\subsubsection{Unitary groups}

The remaining cases in which $G_0=\Un_n(q)$ are the cases in Table \ref{tab:pi1} outlined below 
\begin{itemize}
\item[{\rm (a)}] Case U1: $H$ is of type ${\rm P}_{n/2}$ and $i=2n-2$ with $n=4$ or $6$;
\item[{\rm (b)}] Case U3: $H$ is of type ${\rm GU}_1(q)\perp{\rm GU}_{n-1}(q)$ and $i$ is defined as follows;
 \[
i:=
\begin{cases}
n & n\equiv 0\imod 4\\
n/2 & n\equiv 2\imod 4\\
2n &\mbox{otherwise}
\end{cases}.
\]
\end{itemize}
\vs

\begin{prop}
Theorem \ref{t:mainthrm} holds for case ${\rm U1}$ in Table \ref{tab:pi1}.
\end{prop}
\begin{proof}
Here $H$ is the stabiliser of a totally singular $n/2$-space with $n=4$ or $6$. Note that by Proposition \ref{p:smalldim} we may assume that $q\geqs 3$.
Suppose first that $n=6$. Let $s$ be a primitive prime divisor of $q^{6}-1$ and take an element $x=\hat{x}Z\in G_0$ of order $s$ defined as in \eqref{e:tis}. Then $x$ is a derangement by Lemma \ref{l:tisU}. 
Finally suppose $n=4$. By Proposition \ref{p:smalldim} we may assume $q>8$, so either $q=19$ or $r_6\geqs 13$ by Lemma \ref{l:ppd2a3}. Assume $q\neq 19$, then in $G_0$ there are $(r_6-1)/3\geqs 4$ distinct ${\rm PGU}_4(q)$-classes of elements of order $r_6$. Since $|\Aut(G_0){:}{\rm PGU}_4(q)|=2$ there are at least $(r_6-1)/6\geqs 2$ distinct $G$-classes of elements of order $r_6$ in $G_0$. Finally assume $q=19$ and take $x=\hat{x}Z\in G_0$ to be an element of order 5 (the unique primitive prime divisor of $q^2-1$) such that $\hat{x}=[\mu,\mu^2,\mu^3,\mu^4]$ with $\mu\in\mathbb{F}_{q^2}$ a primitive $5^{{\rm th}}$ root of unity.  Since the eigenvalues of $\hat{x}$ (on $V\otimes \mathbb{F}_{q^2}$) have odd multiplicity, $x$ is a derangement (see \cite[Lemma 4.2.4]{BG_book}). 
\end{proof}

It now remains to deal with case U3 in Table \ref{tab:pi1}. In particular, this leads to the special case appearing in Theorem \ref{t:mainthrm} (see part (i)). 
Recall here we will use $\phi$ to denote a field automorphism of $G_0=\Un_n(q)$ of order $2f$, and note that $\phi^f=\gamma$ is a graph-automorphism. Additionally we note that $\Aut(G_0)=\langle {\rm PGU}_n(q),\phi\rangle$ and we remind the reader that the notation for prime order elements in classical groups was set up in Section \ref{s:conjclass}. We first prove that all prime order derangements must exist in ${\rm PGU}_n(q)$ when $n$ is odd and $H$ is of type ${\rm GU}_1(q)\perp{\rm GU}_{n-1}(q)$.

\begin{lem}\label{l:phiconj}
Let $G_0=\Un_n(q)$ such that $n$ is odd and take $x\in\Aut(G_0)\setminus {\rm PGU}_n(q)$ to be an element of prime order. Then $x$ is ${\rm PGU}_n(q)$-conjugate to $\phi^i$ for some $1\leqs i<2f$.
\end{lem}
\begin{proof}
The group $\Aut(G_0)$ may be split up into a union cosets of ${\rm PGU}_n(q)$, namely $\Aut(G_0)= {\rm PGU}_n(q)\cup {\rm PGU}_n(q)\phi\cup \dots\cup{\rm PGU}_n(q)\phi^{2f-1}$. Thus if $x\in\Aut(G_0)\setminus {\rm PGU}_n(q)$ is an element of prime order $r$, we may assume that $x\in{\rm PGU}_n(q)\phi^i$ such that $|\phi^i|$ has order $r$.
Assume first that $i\neq f$. By \cite[Lemma 3.1.17]{BG_book} every element of prime order in ${\rm PGU}_n(q)\phi^i$ is ${\rm PGU}_n(q)$-conjugate to $\phi^i$, so the result holds. 
Finally assume $i=f$. Then $\phi^i=\gamma$ which implies that $r=2$ and $x$ is a graph automorphism. Note every involutary graph automorphism of $G_0$ is contained in ${\rm PGU}_n(q)\gamma$. Then by \cite[Proposition 3.3.15]{BG_book},  $x$ is ${\rm PGU}_n(q)$-conjugate to $\gamma$. Thus the result follows. 
\end{proof}

\begin{cor}\label{c:phiconj}
Let $G_0=\Un_n(q)$ such that $n$ is odd and take $x\in\Aut(G_0)\setminus {\rm PGU}_n(q)$ to be an element of prime order. Let $V$ denote the natural $G_0$-module. Then $x$ fixes a non-degenerate m-space for $1\leqs m\leqs n$. 
\end{cor}
\begin{proof}
Let $\{v_1,\dots,v_n\}$ be a orthonormal basis for $V=(\mathbb{F}_{q^2})^n$. We recall that the standard field automorphisms are defined as 
$$\phi^i:\sum_j \lambda_jv_j\longmapsto\sum_j \lambda_j^{p^i}v_j.$$
Thus each $\phi^i$ fixes the non-degenerate $m$-space $\langle v_1,\dots, v_m\rangle$ for all $1\leqs i<2f$. Thus the result follows by Lemma \ref{l:phiconj}.
\end{proof}

We are now in a position to handle case U3. 

\begin{prop}\label{p:uniprop}
Theorem \ref{t:mainthrm} holds for case ${\rm U3}$ in Table \ref{tab:pi1}.
\end{prop}
\begin{proof}
Let $x=\hat{x}Z\in H_0$ be an element of order $p$. Then $\hat{x}$ fixes a non-degenerate 1-space $U$ and the non-degenerate $(n-1)$-space, $U^{\perp}$, so $\hat{x}\in {\rm GU}_1(q)\times {\rm GU}_{n-1}(q)$. Therefore $\hat{x}$ is ${\rm GU}_n(q)$-conjugate to $[J_p^{a_p},\dots,J_2^{a_2},J_1^{a_1+1}]$, where $\sum_{t=1}^pta_t=n-1$. Thus a unipotent element in $G_0$ is a derangement if and only if its Jordan form does not contain a Jordan 1-block. This implies that $G_0$ does not contain a derangement of order $p$ if and only if $n$ is odd and $p=2$. 

Assume $n$ is even, or $n$ is odd with $p\geqs 3$. Then by the argument above $G_0$ contains both unipotent and semisimple derangements, so $G$ is not almost elusive. Thus for the remainder of the proof we may assume that $n$ is odd and $p=2$. We recall that in this case any element of order $r_i=r_{2n}$ in $G_0$ is a derangement, where $r_{2n}$ is the unique primitive prime divisor of $q^{2n}-1$ (see Proposition \ref{p:aereduc}).

Suppose first that $n$ is not prime and $(n,q)\neq (9,2)$. Then without loss of generality we can write $n=th$ for positive integers $t$ and $h$ such that $t,h\neq 1$ and $t\geqs 5$. Take $s$ to be a primitive prime divisor of $q^{2t}-1$ and let $x=\hat{x}Z\in G_0$ be an element of order $s$ such that $\hat{x}=[\Lambda^h]$. Then $x$ is a derangement, so $G_0$ contains semisimple derangements of distinct prime order (namely $s$ and $r_{2n}$). Thus $G$ is not almost elusive. 

Next suppose that $(n,q)=(9,2)$. Then 3 is a divisor of $|\Omega|$ and in particular it is the unique primitive prime divisor of $q^2-1$. Take $x=\hat{x}Z\in G_0$ to be an element of order 3 such that $\hat{x}=[\Lambda^3]$ with $\Lambda=\{\mu,\mu^{q^2},\mu^{q^4}\}$ for some $\mu\in\mathbb{F}_{q^6}$ of order $9$. Note that $x\in G_0$ since $(9)_3>(q+1)_3$ (see \cite[Proposition 3.3.3]{BG_book}) and $x$ is a derangement. 

Finally assume $n$ is prime. Note that $r_{2n}$ is also the unique primitive prime divisor of $2^{2nf}-1$, so $r_{2n}=2nfd+1$ for some $d\geqs 1$. Thus $G_0$ contains $(r_{2n}-1)/n=2fd$ distinct ${\rm PGU}_n(q)$-classes of elements of order $r_{2n}$. Since $|\Aut(G_0){:}{\rm PGU}_n(q)|=2f$ there are at least $(r_{2n}-1)/2nf=d$ distinct $G$-classes of elements of order $r_{2n}$ in $G_0$. Therefore $G$ is not almost elusive if $d\geqs 2$, so we may assume $r_{2n}=2nf+1$. By Lemma \ref{l:uniq ppd} either $(n,q,r_{2n})=(5,2,11)$ or $n$ divides $q+1$. The case $(n,q,r_{2n})=(5,2,11)$ has already been handled in Proposition \ref{p:smalldim}, so we may assume that $n$ divides $q+1$. 

We note that the only prime divisors of $|\Omega|$ are $2$, $r_{2n}$ and $n$ (see \cite[Case III of Table 4.1.2]{BG_book} and Remark \ref{r:n=2aj}). Thus these are the only possible primes for prime order derangements in $G$.
Additionally, we note that by Lemma \ref{l:phiconj} and Corollary \ref{c:phiconj} any prime order derangement in $G$ must be contained in ${\rm PGU}_n(q)$. Thus by arguments at the beginning of the proof there are no derangements of order $p=2$ in $G$. 

Note that $n$ is a primitive prime divisor of $q^2-1$. Let $x=\hat{x}Z\in {\rm PGU}_n(q)$ be an element of order $n$. Then by \cite[Proposition 3.3.3]{BG_book}, either $x$ fixes a non-degenerate 1-space, or $x\not\in G_0$ and is such that $\hat{x}=[\Lambda]$ with $\Lambda=\{\mu,\mu^{q^2},\dots,\mu^{q^2(n-1)}\}$ for some $\mu\in\mathbb{F}_{q^{2n}}$ of order $n(q+1)_n$. Thus ${\rm PGU}_n(q)$ contains a derangement of order $n$ and $G_0$ does not. We conclude that if ${\rm PGU}_n(q)\leqs G$ then $G$ is not almost elusive.

Thus we are left to handle the case in which $G\cap {\rm PGU}_n(q)=G_0$. In this case the only possible derangements of prime order in $G$ are the elements of order $r_{2n}=2nf+1$ in $G_0$. Write $G=G_0.J$ where 
$$J\leqs {\rm Out}(G_0)=\langle \ddot{\delta} \rangle {:} \langle \ddot{\phi} \rangle = C_n {:} C_{2f}$$
Then by Corollary \ref{c:unitary case cor}, $G$ is almost elusive if and only if $J$ projects onto $\la \ddot{\phi} \ra $. This completes the proof of the proposition.
\end{proof}

\begin{rmk}
This leaves us with our only potentially infinite family of almost simple almost elusive groups with socle $G_0\in \mathcal{G}$. However, due to the severe number theoretic restrictions in this case (namely $r_{2n}=2nf+1$ being the unique primitive prime divisor of $q^{2n}-1$ with $q=2^f$ and $n$ dividing $q+1$), we anticipate there are in fact no groups that satisfy all the required conditions. See Remarks \ref{r:rmk1} and \ref{r:n=2aj} for more discussion on this.
\end{rmk}

\subsubsection{Orthogonal groups}
To complete the proof of Theorem \ref{t:mainthrm} it remains to handle the orthogonal groups with point stabiliser in $\mathcal{C}_1$. These cases are outlined in Table \ref{tab:orthc1}. We begin with a definition.
\begin{defn}
Let $G_0={\rm P}\Omega^{\epsilon}_n(q)$ with natural module $V$ and let $Q$ denote the associated quadratic form. When $n$ is odd we say that $Q$ is \emph{parabolic} (here $\epsilon=\circ$). When $n$ is even and $Q$ has Witt defect 1 we say $Q$ is \emph{elliptic} (here $\epsilon=-$). Similarly for $n$ even and $Q$ with Witt defect 0 we say $Q$ is \emph{hyperbolic} (here $\epsilon=+$). Additionally we say that a subspace $W$ of $V$ is parabolic (elliptic or hyperbolic) if the restriction of $Q$ to $W$ is parabolic (elliptic or hyperbolic).
\end{defn}

Now we note that if $x\in{\rm P}\Omega^{\epsilon}_n(q)$ is an element of order $r$, such that $r$ is a primitive prime divisor of $q^i-1$ with $i$ even, then $x=\hat{x}Z$ and $\hat{x}$ fixes an orthogonal decomposition of the form 
$$V=U_1\perp\dots \perp U_t \perp {\rm C}_V(\hat{x})$$
where each $U_j$ is an elliptic $i$-space on which $\hat{x}$ acts irreducibly, and ${\rm C}_V(\hat{x})$ is non-degenerate or trivial. We note this is similar to the description of prime order elements in linear groups as discussed in Section \ref{s:conjclass}.

\begin{table}[h!]
\[
\begin{array}{lllll} \hline
\mbox{Case} & G_0 & \mbox{Type of } H & \mbox{Conditions} & i \\ \hline
\mbox{O1} &{\rm P}\Omega^{+}_{n}(q) & {\rm P}_1 &n\equiv 0\imod 4& n-2\\
\mbox{O2} &                                        &  {\rm P}_4 & n=8 &n-2\\
\mbox{O3} &                                        & {\rm Sp}_{n-2}(q) &  n\equiv 2\imod 4 & n/2\\
\mbox{XVI} &                                       &   {\rm Sp}_{n-2}(q) & n\equiv 0\imod 4\\
\mbox{O4} &                                        & {\rm O}_1(q)\perp{\rm O}_{n-1}(q) & n\equiv 2\imod 4 & n/2\\
\mbox{XV} &                                    &  {\rm O}_1(q)\perp{\rm O}_{n-1}(q) & n\equiv 0\imod 4 \\
\mbox{O5} &                                        & {\rm O}^{+}_2(q)\perp{\rm O}^{+}_{n-2}(q)& n\equiv 0\imod 4 & n-2\\
\mbox{O6}&                                         & {\rm O}_2^-(q)\perp{\rm O}^-_{n-2}(q) & n\equiv 0\imod 4& (n-2)/2\\
\mbox{O7} &                                        &  {\rm O}^{-}_2(q)\perp{\rm O}^{-}_{n-2}(q) & n\equiv 2\imod 4 & n/2\\
\mbox{O12} & {\rm P}\Omega^{-}_n(q)  & {\rm P}_1 &   n\equiv 2\imod 4& n\\
\mbox{O13} &                                     &  {\rm Sp}_{n-2}(q) & & n  \\
\mbox{O14} &                                     &  {\rm O}_1(q)\perp{\rm O}_{n-1}(q) & &n \\
\mbox{O15} &                                     &  {\rm O}^{+}_2(q)\perp{\rm O}^{-}_{n-2}(q)&n\equiv 2 \imod 4  &n \\
\mbox{O16} &  \Omega_n(q)  & {\rm P}_1 & n\equiv 1\imod 4 & n-1\\
\mbox{O17} &                      & {\rm O}_1(q)\perp{\rm O}^{+}_{n-1}(q) & &n-1\\
\mbox{O18} &                      & {\rm O}_1(q)\perp{\rm O}^{-}_{n-1}(q) &  n\equiv 3\imod 4 \mbox &(n-1)/2\\
\mbox{XX} &                        &  {\rm O}_1(q)\perp{\rm O}^-_{n-1}(q) & n\equiv 1\imod 4 \\
\mbox{O19} &                      &  {\rm O}^{\epsilon}_{2}(q)\perp{\rm O}_{n-2}(q) & n\equiv 1\imod 4&n-1\\
\hline
\end{array}
\]
\caption{Orthogonal groups with $\mathcal{C}_1$ subgroups from Tables \ref{tab:pi0} and \ref{tab:pi1}}
\label{tab:orthc1}
\end{table}

In Propositions \ref{p:mspace1}, \ref{p:mspace2} and \ref{p:mspace3} we handle the cases in which $H$ is the stabiliser of a totally singular $m$-space for particular $m$. 

\begin{prop}\label{p:mspace1}
Theorem \ref{t:mainthrm} holds for case ${\rm O1}$ in Table \ref{tab:pi1}.
\end{prop}
\begin{proof}
Here $H$ is the stabiliser of a totally singular 1-space and $$|\Omega|=(q^{n/2}-1)(q^{(n-2)/2}+1)/(q-1).$$ Let $s$ be a primitive prime divisor of $q^{n/2}-1$ and note that $n/2$ is even (by Proposition \ref{p:smalldim} we are assuming $(n,q)\neq(12,2)$ so $s$ always exists). Take $x=\hat{x}Z\in G_0$ to be an element of order $s$ such that $\hat{x}=[\Lambda^2]$. Then $x$ does not fix a 1-space, so is a derangement. Therefore $G$ is not almost elusive since $G_0$ contains derangements of order $s$ and of order $r_{n-2}$ (where $r_{n-2}$ denotes the unique primitive prime divisor of $q^{n-2}-1$).
\end{proof}

\begin{prop}\label{p:mspace2}
Theorem \ref{t:mainthrm} holds for case ${\rm O2}$ in Table \ref{tab:pi1} .
\end{prop}
\begin{proof}
Here $(G_0,i)=({\rm P}\Omega^+_8(q),6)$ and $H$ is the stabiliser of a totally singular 4-space. In this case $|\Omega|=(q+1)(q^2+1)(q^3+1)$. Let $s$ be a primitive prime divisor of $q^4-1$ and take $x=\hat{x}Z\in G_0$ to be an element of order $s$ such that $\hat{x}=[\Lambda,I_4]$. Then $x$ is a derangement by Lemma \ref{l:tisSO}. Thus the result follows.   
\end{proof}

\begin{prop}\label{p:mspace3}
Theorem \ref{t:mainthrm} holds for cases ${\rm O12}$ and ${\rm O16}$ in Table \ref{tab:pi1}.
\end{prop}
\begin{proof}
Suppose that $(G_0,H,i)$ is as in case ${\rm O12}$ (respectively, ${\rm O16}$) in Table \ref{tab:pi1} and all relevant conditions hold. Then $H$ is the stabiliser of a totally singular 1-space and $$|\Omega|=(q^{n/2}+1)(q^{(n-2)/2}-1)/(q-1)$$ (respectively, $(q^{n-1}-1)/(q-1)$).

 Assume first that $(n,q)\neq (14,2)$ (note this initial assumption is only necessary for case ${\rm O12}$) and take $s$ to be a primitive prime divisor of $q^{(n-2)/2}-1$ (respectively, $q^{(n-1)/2}-1$). Let $x=\hat{x}Z\in G_0$ be an element of order $s$ such that $\hat{x}=[\Lambda^2,I_2]$ (respectively, $\hat{x}=[\Lambda^2,I_1]$). By \cite[Remark 3.5.5]{BG_book} the 1-eigenspace of $x$ is elliptic (respectively parabolic). Therefore $x$ is a derangement, so $G_0$ contains derangements of order $r_{n}$ (respectively, $r_{n-1}$) and order $s$. Finally assume $(n,q)= (14,2)$ and note that 3 is a primitive prime divisor of $q^2-1$. Take $x=\hat{x}Z\in G_0$ to be an element of order 3 such that $\hat{x}=[\Lambda^7]$. Then $x$ does not fix a 1-space, so is a derangement. Thus the result follows since $s\neq r_{n}$. 
\end{proof}

In the following proposition we handle the cases in which $H$ is the stabiliser of a non-singular 1-space. 

\begin{prop}\label{p:spn-1}
Theorem \ref{t:mainthrm} holds for cases ${\rm O3}$ and ${\rm O13}$ in Table \ref{tab:pi1} and case ${\rm XVI}$ in Table \ref{tab:pi0}.
\end{prop}
\begin{proof}
Here $G_0={\rm P}\Omega_n^{\epsilon}(q)$ and $H$ is the stabiliser of a non-singular 1-space. Recall by Proposition \ref{p:smalldim} we may assume that $q>2$ for $n=8,10$ and $12$. Note that here $q=2^f$ is even and $|\Omega|=q^{n/2-1}(q^{n/2}-\epsilon)$.
Let $r$ be an odd prime divisor of $q^{n/2}-\epsilon$ such that $r$ is a primitive prime divisor of $q^j-1$ for some $j\geqs1$. Note that if $\epsilon=-$ then $j$ divides $n$ but not $n/2$, so $j$ is even and $n/j$ is odd. Similarly if $\epsilon=+$ then $j$ divides $n/2$, so $n/j$ is even. Take $x=\hat{x}Z\in G_0$ to be an element of order $r$ such that 
\[
\hat{x}:=
\begin{cases}
[(\Lambda,\Lambda^{-1})^{(n/2j)}] & \epsilon=+ \mbox{ and } j \mbox{ is odd}\\
[\Lambda^{n/j}] & \mbox{otherwise}
\end{cases}.
\]

Then $x$ does not fix a 1-space, so is a derangement. Therefore we may assume that $q^{n/2}-\epsilon=r^l$ for some odd prime $r$ and $l\geqs 1$. It follows by Lemma \ref{l:btv} that one of the following is satisfied:
\begin{itemize}
\item[{\rm (i)}] $\epsilon=+$ and $r=2^{fn/2}-1$ is a Mersenne prime; or
\item[{\rm (ii)}] $\epsilon =-$ and $r=2^{fn/2}+1$ is a Fermat prime, $n=2^w$ for some $w\geqs 3$ and $f=2^u$ for some $u\geqs 0$.
\end{itemize}

Suppose (i) holds. Then in particular $j=n/2$ is prime and $f=1$, so $\Aut(G_0)={\rm PGO}_n^+(q)$ . It follows that $G_0$ contains $(r-1)/(n/4)=(2^{(n/2+2)}-8)/n\geqs 2$ distinct $G$-classes of derangements of order $r$. Thus we conclude that $G$ is not almost elusive.

Now suppose that (ii) holds. Then $j=n=2^w$ for some $w\geqs 3$ and $f=2^u$ for some $u\geqs 0$. Thus $G_0$ contains $(r-1)/n=2^{(2^k-w)}$ distinct ${\rm PGO}_n^-(q)$-classes of derangements of order $r$, where $k=w+u-1$. Now $|\Aut(G_0){:}{\rm PGO}_n^-(q)|=f$ so we conclude that $G_0$ contains at least $(r-1)/fn=2^{2^k-(k+1)}$ distinct $G$-classes of derangements of order $r$. It is straightforward to check that $k\geqs 3$ and so $2^{2^k-(k+1)}\geqs 2$. Therefore $G$ is not almost elusive.
\end{proof}

In Propositions \ref{p:nd1}, \ref{p:nd2} and \ref{p:nd3} we handle the cases in which $H$ is the stabiliser of a decomposition $V=U\perp W$ of the natural module, where $W$ is a non-degenerate $(n-1)$-dimensional space of type $\epsilon\in\{+,-,\circ\}$. Note that in all of these cases $q=p^f$ is odd. Additionally if $x=\hat{x}Z\in H_0$ is a element of order $p$ then $\hat{x}\in\Omega^{\epsilon}_{n-1}(q)$, and so the Jordan form of $x$ must contain at least one Jordan 1-block.
 
\begin{prop}\label{p:nd1}
Theorem \ref{t:mainthrm} holds for case ${\rm O4}$ in Table \ref{tab:pi1} and case ${\rm XV}$ in Table \ref{tab:pi0}.
\end{prop}
\begin{proof}
In both cases $G_0={\rm P}\Omega_n^+(q)$ with $q$ odd and $H$ is of type ${\rm O}_1(q)\perp {\rm O}_{n-1}(q)$. 
Assume first we are in case O4. Then $n\equiv 2\imod 4$ and any element in $G_0$ of order $r_{n/2}$ is a derangement by Proposition \ref{p:aereduc}. Take $x\in G_0$ to be an element of order $p$ with Jordan form $[J_3^2,J_2^{(n-6)/2}]$ on $V$. Then $x$ is a derangement since it does not contain a Jordan 1-block. Thus $G$ is not almost elusive.

We may assume for the remainder of the proof that we are in case XV, so in particular $n\equiv 0\imod 4$. Take $r$ to be a primitive prime divisor of $q^{n/2}-1$ (note that by Proposition \ref{p:smalldim} $r$ always exists) and let $x=\hat{x}Z\in G_0$ be an element of order $r$ such that $\hat{x}=[\Lambda^2]$. Then $x$ is a derangement. 
Suppose $n\geqs 12$ and take $y\in G_0$ to be an element of order $p$ with Jordan form $[J_3^4,J_2^{(n-12)/2}]$ on $V$. Then $y\not\in H_0$ since the Jordan form does not contain a Jordan 1-block, so we conclude that $G$ is not almost elusive. 
Finally suppose $n=8$. If $p\geqs 5$ then any element in $G_0$ of order $p$ with Jordan form $[J_5,J_3]$ on $V$ is a derangement. Thus we may assume $p=3$ and by Proposition \ref{p:smalldim} $q\neq 3$. Take $s$ to be a primitive prime divisor of $q^2-1$ and let $y=\hat{y}Z\in G_0$ be an element of order $s$ such that $\hat{y}=[\Lambda^4]$. Then $y$ is a derangement, so again $G$ is not almost elusive.
\end{proof}

\begin{prop}\label{p:nd2}
Theorem \ref{t:mainthrm} holds for case ${\rm O14}$ in Table \ref{tab:pi1}.
\end{prop}
\begin{proof}
Here $(G_0,i)=({\rm P}\Omega^{-}_n(q),n)$ and $H$ is of type ${\rm O}_1(q)\perp {\rm O}_{n-1}(q)$. Suppose first that $n>8$. Let $x\in G_0$ be an element of order $p$ with Jordan form $[J_3^2,J_2^{(n-6)/2}]$ if $n\equiv 2\imod 4$ and $[J_3^4,J_2^{(n-12)/2}]$ if $n\equiv 0\imod 4$. Then $x$ is a derangement since there are no Jordan 1-blocks in its Jordan form on $V$. Thus $G_0$ contains both unipotent and semisimple derangements. 
Finally assume $n=8$. By Proposition \ref{p:smalldim} we may assume $q\geqs 5$, so $r_8\geqs 32f+1$ by Lemma \ref{l:ppds4}. Therefore $G_0$ contains $(r_8-1)/8=4f$ distinct ${\rm PGO}^-_n(q)$-classes of elements of order $r_8$ (see \cite[Propositions 3.5.4 and 3.5.8]{BG_book}). Since $|\Aut(G_0){:}{\rm PGO}^-_n(q)|=f$ there are at least $(r_8-1)/8f\geqs 4$ distinct $G$-classes of elements of order $r_8$ in $G_0$, so $G$ is not almost elusive.
\end{proof}

\begin{prop}\label{p:nd3}
Theorem \ref{t:mainthrm} holds for cases {\rm O17} and {\rm O18}  in Table \ref{tab:pi1} and case {\rm XX} in Table \ref{tab:pi0}.
\end{prop}
\begin{proof}
Here $G_0=\Omega_n(q)$ and $H$ is the stabiliser of a non-degenerate $(n-1)$-space of type $\epsilon\in\{+,-\}$.
Assume $n\geqs 9$ when $\epsilon=+$ and take $x=\hat{x}Z\in G_0$ to be an element of order $p$ with the following Jordan form:

\begin{table}[h!]
\[
\begin{array}{l|ll} 
 & n\equiv 1\imod 4 & n\equiv 3\imod 4 \\ \hline
\rule{0pt}{1.10\normalbaselineskip}\epsilon =+& [J_3^3,J_2^{(n-9)/2}] &  [J_3,J_2^{(n-3)/2}]\\
\rule{0pt}{1.10\normalbaselineskip}\epsilon=- &  [J_2^{(n-1)/2},J_1] & [J_3,J_2^{(n-3)/2}]\\
\end{array}
\]
\end{table}
 
By \cite[Proposition 3.5.12]{BG_book} if $\hat{x}\in \Omega^-_{n-1}(q)$ is an element of order $p$ with Jordan form $[J_p^{a_p},\dots,J_1^{a_1}]$, then $a_i>1$ for some odd $i$. Thus $x$ is a derangement.
It follows that in case O17 with $n\geqs 9$ or in case O18, $G_0$ contains both semisimple and unipotent derangements and we are done. 

Suppose we are in case O17 with $n=7$. Then $\epsilon=+$ and by assumption $G_0$ contains semisimple derangements. If $p\geqs 5$ then any element in $G_0$ with Jordan form $[J_3,J_2^2]$ is a derangement, so we may assume $p=3$. By Proposition \ref{p:smalldim} we may additionally assume $q\geqs 9$, so $r_i=r_6\geqs 24f+1$ by Lemma \ref{l:ppd2a3}. Thus continuing in the usual manner, $G_0$ contains $4f$ distinct ${\rm PGO}_7(q)$-classes of order $r_6$, and since $|\Aut(G_0):{\rm PGO}_7(q)|=f$, $G_0$ contains at least 4 distinct $G$-classes of elements of order $r_6$. Thus the result follows.

Finally assume we are as in case XX, then $\epsilon =-$ and $n\equiv 1\imod 4$. Let $s$ be a primitive prime divisor of $q^{(n-1)/2}-1$ 
and take $x=\hat{x}Z\in G_0$ to be an element of order $s$ such that $\hat{x}=[\Lambda^2,I_1]$. Now suppose that $x$ fixes a non-degenerate $(n-1)$-space $W$ of type $\epsilon=-$. Then $\hat{x}$ acts non-trivially on $W$ since $\dim C_V(\hat{x})=1<n-1$, so we obtain a decomposition $W=W_1\perp W_2$ where $W_1$ and $W_2$ are elliptic $(\frac{n-1}{2})$-spaces. This forces $W$ to be a hyperbolic space (of type $\epsilon=+$) which is a contradiction, so we conclude that $x$ is a derangement. Thus $G$ is not almost elusive. 
\end{proof}

The last three propositions deal with the cases in which $H$ is the stabiliser of a decomposition $V=U\perp W$ of the natural module, where $W$ is a non-degenerate $2$-dimensional space of type $\epsilon\in\{+,-\}$.

\begin{prop}
Theorem \ref{t:mainthrm} holds for cases ${\rm O5}$ and ${\rm O6}$ in Table \ref{tab:pi1}.
\end{prop}
\begin{proof}
Here $G_0={\rm P}\Omega_n^+(q)$, $n\equiv 0\imod 4$ and $H$ stabilises a non-degenerate $2$-space of type $\epsilon\in\{+,-\}$. Let $s$ be a primitive prime divisor $q^{n/2}-1$ and note that by Proposition \ref{p:smalldim} we assume $(n,q)\neq(12,2)$, so $s$ always exists. Take $x=\hat{x}Z\in G_0$ to be an element of order $s$ such that $\hat{x}=[\Lambda^2]$. Then $x$ does not fix a 2-space, so $x$ is a derangement. Thus $G_0$ contains derangements of order $r_i$ and order $s$, so the result follows.
\end{proof}

\begin{prop}
Theorem \ref{t:mainthrm} holds for cases ${\rm O7}$ and ${\rm O15}$ in Table \ref{tab:pi1}.
\end{prop}
\begin{proof}
Here $G_0={\rm P}\Omega^{\pm}_n(q)$ with $n\equiv 2\imod 4$ and we note in both cases all elements of $H$ stabilise an $(n-2)$-dimensional non-degenerate elliptic (of type $\epsilon=-$) space. For the case when $(n,q)\neq (14,2)$, we refer the reader to the final paragraph in the proof of Proposition \ref{p:nd3} since the proof here is similar.

Now assume $(n,q)=(14,2)$. Take $(G_0,H,i)$ to be as in case ${\rm O7}$ (respectively case ${\rm O15}$), then any element in $G_0$ of order $r_7=127$ (resp. $r_{14}=43$) is a derangement. The elements $x=\hat{x}Z\in G_0$ of order $r_7$ (resp. $r_{14}$) have the form $\hat{x}=[(\Lambda,\Lambda^{-1})]$ (resp. $[\Lambda]$). Therefore by \cite[Proposition 3.5.4]{BG_book} there are 9 (resp. 3) distinct $\bar{G}={\rm PGO}^+_{14}(2)$ (resp. ${\rm PGO}^-_{14}(2)$)-classes of derangements of order $r_{7}$ (resp. $r_{14}$) in $G_0$. Therefore since $\Aut(G_0)=\bar{G}$ we conclude that $G$ is not almost elusive.
\end{proof}

\begin{prop}
Theorem \ref{t:mainthrm} holds for cases {\rm O19} in Table \ref{tab:pi1}.
\end{prop}
\begin{proof}
Here $G_0=\Omega_n(q)$ with $n\equiv 1\imod 4$ and $H$ is of type ${\rm O}_2^{\epsilon}(q)\perp{\rm O}_{n-2}(q)$ with $\epsilon\in\{+,-\}$. Recall that any element in $G_0$ of order $r_{n-1}$ is a derangement, where $r_{n-1}$ is the unique primitive prime divisor of $q^{n-1}-1$. Let $s$ be a primitive prime divisor of $q^{(n-1)/2}-1$ and take $x=\hat{x}Z\in G_0$ to be an element of order $s$ such that $\hat{x}=[\Lambda^2,I_1]$. Then $x$ does not fix a non-degenerate 2-space of type $\epsilon$, so $G_0$ contains derangements of order $r_{n-1}$ and $s$, implying that $G$ is not almost elusive. 
\end{proof}

This completes the proof of Theorem \ref{t:mainthrm} for the subspace subgroups. In particular, in view of Propositions \ref{p:smalldim} and \ref{p:allnonsub}, this completes the proof of Theorem \ref{t:mainthrm} entirely.

\end{document}